\documentclass{article}

\usepackage[english]{babel}

\usepackage[letterpaper,top=2cm,bottom=2cm,left=3cm,right=3cm,marginparwidth=1.75cm]{geometry}

\usepackage{amsfonts,amsmath,amssymb,amsbsy,amsthm,enumerate}

\usepackage{graphicx}
\usepackage[colorlinks=true, allcolors=blue]{hyperref}
\usepackage{tikz}
\usetikzlibrary{positioning, backgrounds, fit}
\usepackage[font=small,labelfont=bf, width=.85\textwidth]{caption}
\usepackage{authblk}
\usepackage{mathrsfs}
\usepackage{blindtext}
\usepackage{verbatim}
\usepackage{multicol}
\usepackage{multirow}
\usepackage{array}
\usepackage{booktabs}
\usepackage{amsmath}
\usepackage{booktabs}
\usepackage{pgffor}
\usepackage{xparse}
\usepackage[table]{xcolor}

\newtheorem{theorem}{Theorem}
\newtheorem{conjecture}{Conjecture}

\newtheorem{claim}{Claim}

\newtheorem*{lemma*}{Lemma}
\newtheorem*{claim*}{Claim}

\newtheorem{apptheorem}{Theorem}

\newenvironment{proofc}{{\noindent \textit{Proof of Claim.}}}
{\hfill \Large{$\diamond$}}

\title{Coloring graphs with independence number two and no odd clique immersions}

\author{Henry Echeverría\footnote{ Instituto de Ingenier\'ia Matem\'atica-CIMFAV, Universidad de Valpara\'iso, Chile.
Email: {\tt henry.echeverria@postgrado.uv.cl}. Supported by ANID BECAS/DOCTORADO NACIONAL 21231147.
}
\qquad
Jessica McDonald\footnote{Auburn University, Department of Mathematics and Statistics, Auburn USA.
  Email: {\tt mcdonald@auburn.edu}.   
	Supported in part by Simons Foundation Grant \#845698 and NSF grant DMS-2452103.}
}

\date{}

\begin{document}
\maketitle

\begin{abstract} We study the chromatic number of graphs that exclude a clique as a strong odd immersion and have independence number two. Given a graph $G$ and $t\in\mathbb{Z}^+$, we prove that if $\alpha(G)\leq 2$ and $G$ has no strong odd $K_t$-immersion, then $\chi(G)\leq \lceil \frac{3(t-1)}{2}\rceil$. 
\end{abstract}

\section{Introduction}

In this paper every graph is simple, that is, it contains no loops or parallel edges; we refer the reader to \cite{Diestel} for terms not defined here.

Given graphs $G, H$, we say that $G$ contains the graph $H$ as an \emph{immersion} ($G$ has an \emph{$H$-immersion}) if there exists an injective function $\varphi\colon V(H)\rightarrow V(G)$ and a set of pairwise edge-disjoint paths $\mathcal{P}$ in $G$ such that for every $uv\in E(H)$, there exists $P_{uv}\in\mathcal{P}$ with endpoints $\varphi(u)$ and~$\varphi(v)$. The vertices in $\varphi(V(H))$ are called the \emph{terminals} of the immersion; if no path in $\mathcal{P}$ contains a terminal as an interior vertex then we say that the immersion of $H$ in $G$ is \emph{strong}. If all paths in $\mathcal{P}$ have odd length, then we say that the immersion of $H$ in $G$ is \emph{totally odd} (or just \emph{odd}). 

Clique immersions are connected to coloring via a series of long-standing conjectures that we shall discuss shortly. Of particular interest is the case when the \emph{independence number} $\alpha(G)$ (the maximum size of an independent set in graph $G$) is two. In our main result, which follows, the \emph{chromatic number} $\chi(G)$ of a graph $G$ is the least $k$ such that $G$ is \emph{$k$-colorable}, i.e., has an assignment of $1, 2, \ldots, k$ to its vertices such that adjacent vertices receive different colors.

\begin{theorem} \label{thm:main}  Let $t\in\mathbb{Z}^+$ and let $G$ be a graph. If $\alpha(G)\leq 2$ and $G$ has no strong odd $K_t$-immersion, then $\chi(G)\leq \lceil \tfrac{3}{2}(t-1)\rceil$.\end{theorem}

Theorem \ref{thm:main} is an immersion-analog of a result of K\"{u}hn, Sauermann, Steiner, and Wigderson (Proposition~1.5 in \cite{KuhnSauermanSteinerWigderson2024}); their result is stated identically to ours except with  ``odd $K_t$-minor''  replacing  ``strong odd $K_t$-immersion''. For K\"{u}hn, Sauerman, Steiner, and Wigderson,  the constant $\tfrac{3}{2}$ is optimal, as shown via their disproof of the so-called ``Odd-Hadwiger Conjecture'' in their same paper.  We however conjecture that the constant $\tfrac{3}{2}$ in Theorem \ref{thm:main} is not best possible, and we discuss this in more detail following our proof of Theorem \ref{thm:main} in the next section.

Theorem \ref{thm:main} makes a positive contribution towards the following difficult conjecture.

\begin{conjecture}\label{conj:strongOdd} Let $t\in\mathbb{Z}^+$ and let $G$ be a graph. If $G$ has no strong odd $K_t$-immersion, then $\chi(G)<t$.
\end{conjecture}

When the word ``odd'' is removed from Conjecture \ref{conj:strongOdd}, the conjecture is due to either Lescure and Meyniel \cite{LM89} or Abu-Khzam and Langston \cite{AbuLangston}, depending on whether you also remove the word ``strong'' or not. The word ``odd'' was first added by Churchley \cite{churchley2017odd}, and is significant in that it is an extension to arbitrarily dense graph classes; it is easy to see that no complete bipartite graph can contain an odd $K_3$-immersion, while the class of graphs without a $K_t$-immersion are known to have bounded minimum degree \cite{DeVosKMMO2010}. The full form of Conjecture \ref{conj:strongOdd} was first put forward by Jim\'enez, Quiroz and Thraves Caro~\cite{JimenezQuirozThraves}.

Conjecture \ref{conj:strongOdd} is known for $t\leq 4$  (Zang \cite{Z98}, Thomassen \cite{Thomassen}) and for some very special graph classes (eg. line graphs of constant multiplicity multigraphs \cite{JimenezQuirozThraves}, Kneser graphs \cite{Echeverria1}, certain graph products \cite{Echeverria2}). The best general bound towards Conjecture \ref{conj:strongOdd} is due to Kawarabayashi~\cite{K13}, who showed (as a consequence of his work on subdivisions) that every such $G$ has $\chi(G)\leq 79t^2/4$. Very recently, McFarland \cite{McFarland} has shown that if the word ``strong'' is removed from Conjecture \ref{conj:strongOdd} one can get a bound of $\chi(G)<\mathcal{O}(t)$, albeit with large constants ($\chi(G) \leq 700,000t + 32,000,000$). When the word ``odd'' is removed from Conjecture \ref{conj:strongOdd}, results improve: up to $t=7$ is known  (DeVos, Kawarabayashi, Mohar, and Okamura \cite{DKMO}), and the bound $\chi(G)<3.54t + 4$ can be obtained (Gauthier, Le and Wollan \cite{GauthierLW2019}).

The class of graphs $G$ with no $K_t$-immersion and with $\alpha(G)\leq 2$ has actually been well-studied, particularly in the context of the following conjecture.

\begin{conjecture}\label{conj:vergara}\emph{(Vergara \cite{Vergara2017})} Let $G$ be an $n$-vertex graph with $\alpha(G)\leq 2$. If $G$ has no $K_t$-immersion, then $n < 2t$. 
\end{conjecture}

Note that Conjecture \ref{conj:vergara} is a weakening of Conjecture \ref{conj:strongOdd} in the case of $\alpha(G)\leq 2$ (and with both ``strong'' and ``odd'' removed), since for any $n$-vertex graph $G$ with $\alpha(G)\leq 2$, we know $\chi(G)\geq \tfrac{n}{\alpha(G)}=\tfrac{n}{2}$ (and hence $\chi(G)<t$ implies $n<2t$). Vergara \cite{Vergara2017} proved Conjecture \ref{conj:vergara} if $n < 2t$ is replaced by $n < 3t$, and this was improved to $2 \lfloor \frac{n}{5} \rfloor <t$ by 
Gauthier, Le, and Wollan \cite{GauthierLW2019}. Recently, Botler, Fernandes, Lintzmayer, Lopes, Mishra, Netto and Sambinelli \cite{BotlerFernandesLintzmayerLopesMishraNettoSambinelli2025} proved Conjecture \ref{conj:vergara} for graphs with bounded maximum degree in terms of $n$.  These results towards Conjecture \ref{conj:vergara}, with the exception of Vergara's own, were proved \emph{with} the extra conditions of strong and odd for the immersion, and  Quiroz~\cite{Q1} pointed out to us that Vergara's proof can be improved in this way as well (using techniques from a paper of Bustamante, Quiroz, Stein, Zamora \cite{BustamanteQSZ2021}). We include a proof of this improved result of Vergara in our appendix for completeness; we will use this result (Theorem \ref{thm:VergStrong}) in our proof of Theorem \ref{thm:main}.

One hope for generalizing Theorem \ref{thm:main} would be to get bound on $\chi(G)$ involving a fixed $\alpha(G)$. Since $\chi(G)\leq n$, a result of this type is implied by work of Bustamante, Quiroz, Stein, and Zamora \cite{BustamanteQSZ2021} who proved the following: \emph{If $t\in\mathbb{Z}^+$ and $G$ is an $n$-vertex graph with no strong odd $K_t$-immersion, then $\lfloor \frac{n}{2.25(\alpha(G)-1)}\rfloor\leq t$.} When $\alpha(G) \geq 4$ we cannot really improve this, but when $\alpha(G)=3$ their result says that $\lfloor \frac{n}{2.25(2)}\rfloor\leq t$ and we can make a modest improvement to $\chi(G)\leq 4(t-1)$. We include this argument in the following section along with our main proof.

\section{Proof of Theorem \ref{thm:main}}

We say a graph $G$ is \emph{$k$-vertex-critical} if it has chromatic number $k$ but $G-v$ is  $(k-1)$-colorable for every $v\in V(G)$. Given a graph $G$ and vertex sets $X, Y\subseteq V(G)$, we let $G[X]$ be the subgraph of $G$ induced by $X$ and we let $G[X,Y]$ be the bipartite subgraph of $G$ induced by the bipartition $(X, Y)$. We will use the following classic result of Gallai.

\begin{theorem} \emph{(Gallai \cite{Gallai1963}) } \label{Gallai}
Let $k \geq 2$ be an integer, and let $G$ be a $k$-vertex-critical graph on $n$ vertices. If $n \leq 2k - 2$, then there exists a partition of $V(G)$ into two non-empty sets $X_1,X_2$ such that $G[X_1, X_2]$ is a complete bipartite graph. 
\end{theorem}

We are now ready to prove our main result. As mentioned in the introduction, this is an immersion-analog of a result of K\"{u}hn, Sauerman, Steiner, and Wigderson \cite{KuhnSauermanSteinerWigderson2024}, and indeed our proof starts off very similar to theirs. After Claim \ref{claim:case2} however, there is a divergence and we need new ideas to find the immersions we are looking for. 

\begin{proof}[Proof of Theorem \ref{thm:main}]
 Let $G$ be a vertex-minimal counterexample for Theorem \ref{thm:main}, so in particular
$\alpha=\alpha(G)\leq 2$, $G$ has no strong odd $K_t$-immersion, and $k:= \chi(G)>  \lceil \tfrac{3}{2}(t-1)\rceil$. By Theorem \ref{thm:VergStrong} we get $\frac{n}{3} \leq \lceil \frac{n}{3}\rceil \leq t-1$. This implies that $\frac{n}{2} \leq \frac{3}{2}(t-1) \leq \lceil \frac{3}{2}(t-1)\rceil \leq k-1$ and $n\leq 2k-2$.

Note that $G$ is $k$-critical by minimality. Since $n\leq 2k-2$ we apply Theorem \ref{Gallai} and get a partition of $V(G)$ into non-empty sets $X_1, X_2$ such that $x_1 x_2 \in E(G)$ for all $x_1 \in X_1$ and $x_2 \in X_2$. Let $G_1 := G[X_1]$ and $G_2 := G[X_2]$, and note that we have $\chi(G) = \chi(G_1) + \chi(G_2)$. For $i \in \{1, 2\}$, let $t_i$ denote the largest positive integer such that $G_i$ has $K_{t_i}$ as a strong odd immersion; note that $G$ has a strong odd immersion of $K_{t_1+t_2}$, so $t_1+t_2\leq t-1$. Since $\alpha_i:=\alpha(G_i)\leq \alpha\leq 2$, by minimality we get that 
$\chi(G_i ) \leq \lceil \frac{3}{2}t_i \rceil$ for $i \in \{1, 2\}$. In particular, this means that
\begin{align}\label{eq:t1t2}
         \chi(G)=\chi(G_1)+\chi(G_2) \leq \left \lceil\tfrac{3}{2}t_1\right\rceil  + \left \lceil\tfrac{3}{2}t_2\right\rceil \leq \tfrac{3}{2}t_1  + \tfrac{3}{2}t_2 + \tfrac{1}{2} +\tfrac{1}{2} \leq \tfrac{3}{2}(t-1)+1\leq \chi(G).
    \end{align}
Note that (\ref{eq:t1t2}) implies that $t_1, t_2$ are both odd, $t_1+t_2=t-1$, and $k=\chi(G)=\tfrac{3}{2}(t-1)+1$. It also implies that $\chi(G_i)=\lceil\tfrac{3}{2}t_i\rceil=\tfrac{3}{2}t_i+\tfrac{1}{2}$ for $i=1, 2$. Since $G$ is $k$-vertex-critical we know that $\delta(G)\geq k-1$, and since we may assume that $t\geq 5$ (as otherwise the result holds) we also know that $\delta(G) \geq k-1=\tfrac{3}{2}(t-1)\geq t+1$.

For $i \in \{1, 2\}$, let $M_i \subseteq X_i$ be chosen as an inclusion-wise minimal subset of $X_i$ such that $G[M_i]$ still contains $K_{t_i}$ as a strong odd immersion.

\begin{claim} \label{claim:case2}        For every $i \in \{1, 2\}$, $y \in M_i$, the set $(X_i - M_i) \cup  \{y\}$ is not independent. 
    \end{claim}
    \begin{proofc}
        Suppose not, that is, suppose without loss of generality that there exists $y\in M_1$ such that $(X_1 - M_1) \cup  \{y\}$ is independent. Let $G_1'=G_1[M_1-y]$ and note that $G_1'$ has no  strong odd $K_{t_1}$-immersion so by minimality $\chi(G_1')\leq \lceil \frac{3}{2}(t_1-1)\rceil$. We can extend this to a coloring of $G$ by assigning one new color to the independent set $(X_1 - M_1) \cup  \{y\}$ and using $\chi(G_2)$ new colors on the rest of the vertices. So
    \begin{align} \label{ineq:1}
        \chi(G) = \chi(G_1')+1+\chi(G_2)
        \leq \lceil \tfrac{3}{2}(t_1-1)\rceil+ 1+ \lceil \tfrac{3}{2}(t_2)\rceil 
    \end{align}
Since $t_1$ and $t_2$ are both odd and $t_1+t_2 = t-1$,  inequality \ref{ineq:1} implies the following:
\begin{align*} 
        \chi(G) &\leq  \tfrac{3}{2}(t_1-1)+ 1+  \tfrac{3}{2}(t_2) + \tfrac{1}{2} = \tfrac{3}{2}(t_1+t_2) =\tfrac{3}{2}(t-1) <\chi(G),
    \end{align*}
contradiction.
\end{proofc}

 Let $Z_i = X_i-M_i$ for $i\in\{1,2\}$. Recall that for $i\in\{1,2\}$, $G[M_i]$ contains $K_{t_i}$ as a strong odd immersion; fix one such immersion and partition the set $M_i$ into two sets $T_i, A_i$ where $T_i$ are the terminals and $A_i$ are the non-terminals (note that $A_i, T_i$ are disjoint since the immersion is strong). We can observe that $T_1\cup T_2$ are the terminals of a strong odd $K_{t-1}$-immersion in $G$. 
For $i\in\{1,2\}$, we additionally partition $A_i$ into two (possibly empty) sets $B_i, C_i$ so that vertices in $B_i$ have no neighbors in $Z_i$, and $C_i=A_i-B_i$. 

\begin{claim}\label{claim:Ztwo} $|Z_1|, |Z_2|\geq 2$.
\end{claim}
\begin{proofc} There must be at least one element in each of $Z_1, Z_2$, otherwise $X_i-M_i=\emptyset$ for some $i$ and we get a contradiction to Claim \ref{claim:case2}. If $Z_i=\{v\}$, then by Claim \ref{claim:case2} the vertex  $v$ must be adjacent to every vertex in $M_i$, and in particular $M_1\cup M_2 \cup \{v\}$ contains a strong odd $K_{t_1+t_2+1}$-immersion, which is a contradiction since $t_1+t_2+1=t$.
\end{proofc}

We will now work to build a set of edge-disjoint odd-length paths $\mathcal{P}_v$ from some vertex $v\in Z_1\cup Z_2$ to distinct vertices in $T_1\cup T_2$ which are internally vertex-disjoint from $T_1\cup T_2$, and edge-disjoint from $G_1[M_1], G_2[M_2]$ and $G[T_1, T_2]$. If we can build such a set $\mathcal{P}_v$ with $|\mathcal{P}|=|T_1\cup T_2|$ then this would give us a strong odd $K_t$-immersion in $G$, and hence our desired contradiction. We initialize $\mathcal{P}_v=\emptyset$ and say that a path $P$ is \emph{acceptable (for $\mathcal{P}_v$)} if we can add $P$ to $\mathcal{P}_v$ while maintaining all the desired properties for $\mathcal{P}_v$. 
We say that we have \emph{solved} edge $vw$ for $\mathcal{P}_v$ if we have added a path to $\mathcal{P}_v$ that starts with the edge $vw$. Note that we get our desired result if we are able to solve at least $|T_1\cup T_2|=t-1$ edges incident to some $v$.

Given any $v\in Z_i$ ($i=1, 2$), we immediately add to $\mathcal{P}_v$ any edge from $v$ to $T_1\cup T_2$; let $T^*\subseteq T_1\cup T_2$ consist of every vertex $t\in T_i$ for which we have not yet been able to define an acceptable path from $v$ to $t$ to add to $\mathcal{P}_v$. Observe that $T^*\subseteq T_i$ for any $v\in Z_i$, $i=1, 2$.

\begin{claim} \label{claim:cliques1} $B_1,B_2\neq \emptyset$, and $G[B_1], G[B_2], G[Z_1], G[Z_2]$ are all cliques. 
\end{claim}
\begin{proofc}
We denote $v\sim u$ (or $u\sim v$) if $vu\in E(G)$. Suppose first that $G[Z_1], G[Z_2]$ are not cliques; suppose without loss that $ v,w \in Z_1$ and $v,w$ are not adjacent. Observe that since $\alpha=2$, $w\sim t$ for every $t\in T^*$. So, for any vertex $u\in A_2\cup Z_2$, since $u\sim v,w$, we get an odd path $(v, u, w, t)$ for any $t\in T^*$ (see the top left picture in Figure \ref{fig:claim3cliques}). So we know that $|A_2\cup Z_2|<|T^*|$, otherwise we get an strong odd $K_{t}$-immersion in $G$.

We have just said that we can solve every edge from $v$ to $A_2\cup Z_2$; update $\mathcal{P}_v$ and $ T^*$ accordingly. Let us now show that we can also solve every edge from $v$ to $A_1\cup Z_1\setminus\{v\}$. To this end, fix any $z\in Z_2$. Then for any vertex $s\in A_1\cup Z_1\setminus \{v\}$, since $s\sim v,z$, and since $z\sim t$ for all $t\in T^*$, we get an odd path $(v, s, z, t)$ for any $t\in T^*$. See the top right picture in Figure \ref{fig:claim3cliques}.

We have now solved every edge incident to $v$. Since $\delta(G)\geq t+1$, this is a contradiction. So, $G[Z_1], G[Z_2]$ are cliques.

Suppose now that one of $B_1, B_2$ is empty; say without loss that $B_2=\emptyset$.
Fix a vertex $v\in Z_1$. We will show that we can solve every edge incident to $v$ except for $vz$, which will yield our desired contradiction. For a vertex $s\in N(v)$ such that $s\in A_1\cup Z_1$ , since $s\sim v$, $s\sim z$, and since $z\sim t$ for all $t\in T^*$, we get an odd path $(v, s, z, t)$ for any $t\in T^*$. See the bottom-left picture in Figure \ref{fig:claim3cliques}.

We have just said that we can solve every edge from $v$ to $A_1\cup Z_1\setminus \{v\}$; update $\mathcal{P}_v$ and $T^*$ accordingly. Let us now show that we can also solve every edge from $v$ to $A_2\cup Z_2$. First suppose that $v\sim u$ for some $u\in A_2$. Since $G[A_{2}] = G[C_2]$, every vertex in $u\in A_2$ has a neighbor in $z_u\in Z_2$, which is in turn adjacent to every vertex in $T^*$. So we get an odd path $(v, u, z_u, t)$ for any $t\in T^*$. On the other hand, if $v\sim u'$ for some $u'\in Z_2\setminus \{z\}$, then since $G[Z_2]$ is a clique, $u'\sim z$, and $z$ is adjacent to every vertex in $T^*$. So we get an odd path $(v, u', z, t)$ for any $t\in T^*$, as we claimed. See the bottom-right picture in Figure \ref{fig:claim3cliques}.

We now know that $B_i, B_j\neq \emptyset$. For $i\in \{1, 2\}$ if $|B_i|=1$ then $G[B_i]$ is trivially a clique. Otherwise, if $G[B_i]$ is not a clique, then there are two non-adjacent vertices in $B_i$ and since they are not adjacent to any vertices in $Z_i\neq\emptyset$, we get an independent set of size 3, contradiction. So $G[B_1], G[B_2]$ are cliques.\end{proofc}

\begin{figure}[h!]
    \centering
    \makebox[\textwidth][c]{%
    \resizebox{1\linewidth}{!}{%
    \begin{tabular}{c c}
        \tikzset{every picture/.style={line width=0.75pt}} 

\begin{tikzpicture}[x=0.75pt,y=0.75pt,yscale=-1,xscale=1]

\draw   (61,152.83) .. controls (61,96.78) and (90.03,51.33) .. (125.83,51.33) .. controls (161.64,51.33) and (190.66,96.78) .. (190.66,152.83) .. controls (190.66,208.89) and (161.64,254.33) .. (125.83,254.33) .. controls (90.03,254.33) and (61,208.89) .. (61,152.83) -- cycle ;
\draw    (63.01,130.53) -- (190,131) ;
\draw    (61,152.83) -- (189.78,152.28) ;
\draw   (248,145.83) .. controls (248,89.78) and (277.03,44.33) .. (312.83,44.33) .. controls (348.64,44.33) and (377.66,89.78) .. (377.66,145.83) .. controls (377.66,201.89) and (348.64,247.33) .. (312.83,247.33) .. controls (277.03,247.33) and (248,201.89) .. (248,145.83) -- cycle ;
\draw    (255,100.33) -- (366.66,100.1) ;
\draw    (248.67,139.95) -- (377.67,139.95) ;
\draw  [fill={rgb, 255:red, 74; green, 99; blue, 226 }  ,fill opacity=1 ] (158,175.04) .. controls (158,172.62) and (159.96,170.66) .. (162.39,170.66) .. controls (164.81,170.66) and (166.78,172.62) .. (166.78,175.04) .. controls (166.78,177.47) and (164.81,179.43) .. (162.39,179.43) .. controls (159.96,179.43) and (158,177.47) .. (158,175.04) -- cycle ;
\draw    (125.83,51.33) -- (126.5,130.77) ;
\draw  [fill={rgb, 255:red, 0; green, 0; blue, 0 }  ,fill opacity=1 ] (123,219.04) .. controls (123,216.62) and (124.96,214.66) .. (127.39,214.66) .. controls (129.81,214.66) and (131.78,216.62) .. (131.78,219.04) .. controls (131.78,221.47) and (129.81,223.43) .. (127.39,223.43) .. controls (124.96,223.43) and (123,221.47) .. (123,219.04) -- cycle ;
\draw    (166.78,175.04) -- (276,154) ;
\draw  [fill={rgb, 255:red, 0; green, 0; blue, 0 }  ,fill opacity=1 ] (271.61,154) .. controls (271.61,151.58) and (273.58,149.61) .. (276,149.61) .. controls (278.42,149.61) and (280.39,151.58) .. (280.39,154) .. controls (280.39,156.42) and (278.42,158.39) .. (276,158.39) .. controls (273.58,158.39) and (271.61,156.42) .. (271.61,154) -- cycle ;
\draw    (127.39,219.04) -- (276,154) ;
\draw    (146,112) .. controls (142.01,123.53) and (123.01,197.53) .. (127.39,219.04) ;
\draw  [fill={rgb, 255:red, 0; green, 0; blue, 0 }  ,fill opacity=1 ] (146,112) .. controls (146,109.58) and (147.96,107.61) .. (150.39,107.61) .. controls (152.81,107.61) and (154.78,109.58) .. (154.78,112) .. controls (154.78,114.42) and (152.81,116.39) .. (150.39,116.39) .. controls (147.96,116.39) and (146,114.42) .. (146,112) -- cycle ;

\draw (36,70.4) node [anchor=north west][inner sep=0.75pt]    {$T_{1}$};
\draw (228,64.4) node [anchor=north west][inner sep=0.75pt]    {$T_{2}$};
\draw (30,130.4) node [anchor=north west][inner sep=0.75pt]    {\large$A_{1}$};
\draw (217,108.4) node [anchor=north west][inner sep=0.75pt]    {\large$A_{2}$};
\draw (33,182.4) node [anchor=north west][inner sep=0.75pt]    {\large$Z_{1}$};
\draw (222,201.4) node [anchor=north west][inner sep=0.75pt]    {\large$Z_{2}$};
\draw (146,167.4) node [anchor=north west][inner sep=0.75pt]    {\large$v$};

\draw (131,72.4) node [anchor=north west][inner sep=0.75pt]    {\large$T^*$};
\draw (104,208.4) node [anchor=north west][inner sep=0.75pt]    {\large$w$};
\draw (278,157.4) node [anchor=north west][inner sep=0.75pt]    {\large$u$};
\draw (158,104.4) node [anchor=north west][inner sep=0.75pt]    {\large$t$};

\end{tikzpicture} & \tikzset{every picture/.style={line width=0.75pt}} 

\begin{tikzpicture}[x=0.75pt,y=0.75pt,yscale=-1,xscale=1]

\draw   (61,152.83) .. controls (61,96.78) and (90.03,51.33) .. (125.83,51.33) .. controls (161.64,51.33) and (190.66,96.78) .. (190.66,152.83) .. controls (190.66,208.89) and (161.64,254.33) .. (125.83,254.33) .. controls (90.03,254.33) and (61,208.89) .. (61,152.83) -- cycle ;
\draw    (63.01,130.53) -- (190,131) ;
\draw    (61,152.83) -- (189.78,152.28) ;
\draw   (248,149) .. controls (248,94.69) and (277.03,50.66) .. (312.83,50.66) .. controls (348.64,50.66) and (377.66,94.69) .. (377.66,149) .. controls (377.66,203.31) and (348.64,247.33) .. (312.83,247.33) .. controls (277.03,247.33) and (248,203.31) .. (248,149) -- cycle ;
\draw    (256,100.33) -- (367.66,100.1) ;
\draw    (248.67,139.95) -- (377.67,139.95) ;
\draw  [fill={rgb, 255:red, 0; green, 75; blue, 245 }  ,fill opacity=1 ] (121,175.04) .. controls (121,172.62) and (122.96,170.66) .. (125.39,170.66) .. controls (127.81,170.66) and (129.78,172.62) .. (129.78,175.04) .. controls (129.78,177.47) and (127.81,179.43) .. (125.39,179.43) .. controls (122.96,179.43) and (121,177.47) .. (121,175.04) -- cycle ;
\draw    (125.83,51.33) -- (126.5,130.77) ;
\draw  [fill={rgb, 255:red, 0; green, 0; blue, 0 }  ,fill opacity=1 ] (121,215.04) .. controls (121,212.62) and (122.96,210.66) .. (125.39,210.66) .. controls (127.81,210.66) and (129.78,212.62) .. (129.78,215.04) .. controls (129.78,217.47) and (127.81,219.43) .. (125.39,219.43) .. controls (122.96,219.43) and (121,217.47) .. (121,215.04) -- cycle ;
\draw    (129.78,175.04) -- (158.66,204.47) ;
\draw  [fill={rgb, 255:red, 0; green, 0; blue, 0 }  ,fill opacity=1 ] (154.27,204.47) .. controls (154.27,202.05) and (156.23,200.08) .. (158.66,200.08) .. controls (161.08,200.08) and (163.04,202.05) .. (163.04,204.47) .. controls (163.04,206.9) and (161.08,208.86) .. (158.66,208.86) .. controls (156.23,208.86) and (154.27,206.9) .. (154.27,204.47) -- cycle ;
\draw    (163.04,204.47) -- (321.61,190) ;
\draw    (144.06,110.06) .. controls (159.43,134.18) and (285.04,161.08) .. (326,185.61) ;
\draw  [fill={rgb, 255:red, 0; green, 0; blue, 0 }  ,fill opacity=1 ] (139.67,105.67) .. controls (139.67,103.24) and (141.63,101.28) .. (144.06,101.28) .. controls (146.48,101.28) and (148.44,103.24) .. (148.44,105.67) .. controls (148.44,108.09) and (146.48,110.06) .. (144.06,110.06) .. controls (141.63,110.06) and (139.67,108.09) .. (139.67,105.67) -- cycle ;
\draw  [fill={rgb, 255:red, 0; green, 0; blue, 0 }  ,fill opacity=1 ] (321.61,190) .. controls (321.61,187.58) and (323.58,185.61) .. (326,185.61) .. controls (328.42,185.61) and (330.39,187.58) .. (330.39,190) .. controls (330.39,192.42) and (328.42,194.39) .. (326,194.39) .. controls (323.58,194.39) and (321.61,192.42) .. (321.61,190) -- cycle ;
\draw [line width=1.5]    (129.78,175.04) -- (175,141.72) ;
\draw  [fill={rgb, 255:red, 0; green, 0; blue, 0 }  ,fill opacity=1 ] (170.61,141.72) .. controls (170.61,139.29) and (172.57,137.33) .. (175,137.33) .. controls (177.42,137.33) and (179.39,139.29) .. (179.39,141.72) .. controls (179.39,144.14) and (177.42,146.11) .. (175,146.11) .. controls (172.57,146.11) and (170.61,144.14) .. (170.61,141.72) -- cycle ;
\draw [line width=1.5]    (326,190) -- (175,141.72) ;
\draw [line width=1.5]    (326,190) .. controls (332.27,176.77) and (165.66,95.77) .. (169.44,104.67) ;
\draw  [fill={rgb, 255:red, 0; green, 0; blue, 0 }  ,fill opacity=1 ] (160.67,104.67) .. controls (160.67,102.24) and (162.63,100.28) .. (165.06,100.28) .. controls (167.48,100.28) and (169.44,102.24) .. (169.44,104.67) .. controls (169.44,107.09) and (167.48,109.06) .. (165.06,109.06) .. controls (162.63,109.06) and (160.67,107.09) .. (160.67,104.67) -- cycle ;

\draw (36,70.4) node [anchor=north west][inner sep=0.75pt]    {\large$T_{1}$};
\draw (229,71.4) node [anchor=north west][inner sep=0.75pt]    {\large$T_{2}$};
\draw (31,123.4) node [anchor=north west][inner sep=0.75pt]    {\large$A_{1}$};
\draw (223,108.4) node [anchor=north west][inner sep=0.75pt]    {\large$A_{2}$};
\draw (34,174.4) node [anchor=north west][inner sep=0.75pt]    {\large$Z_{1}$};
\draw (226,168.4) node [anchor=north west][inner sep=0.75pt]    {\large$Z_{2}$};
\draw (98,168.4) node [anchor=north west][inner sep=0.75pt]    {\large$v$};
\draw (135,70.4) node [anchor=north west][inner sep=0.75pt]    {\large$T^{*}$};
\draw (94,205.4) node [anchor=north west][inner sep=0.75pt]    {\large$w$};
\draw (138,195.4) node [anchor=north west][inner sep=0.75pt]    {\large$s'$};
\draw (134.06,107.68) node [anchor=north west][inner sep=0.75pt]    {\large$t$};
\draw (328,193.4) node [anchor=north west][inner sep=0.75pt]    {\large$z$};
\draw (167.06,108.07) node [anchor=north west][inner sep=0.75pt]    {\large$t'$};
\draw (158,131.4) node [anchor=north west][inner sep=0.75pt]    {\large$s$};

\end{tikzpicture} \\
        &\\
        \tikzset{every picture/.style={line width=0.75pt}} 

\begin{tikzpicture}[x=0.75pt,y=0.75pt,yscale=-1,xscale=1]

\draw   (61,152.83) .. controls (61,96.78) and (90.03,51.33) .. (125.83,51.33) .. controls (161.64,51.33) and (190.66,96.78) .. (190.66,152.83) .. controls (190.66,208.89) and (161.64,254.33) .. (125.83,254.33) .. controls (90.03,254.33) and (61,208.89) .. (61,152.83) -- cycle ;
\draw    (63.01,130.53) -- (190,131) ;
\draw    (61,152.83) -- (189.78,152.28) ;
\draw   (248,149) .. controls (248,94.69) and (277.03,50.66) .. (312.83,50.66) .. controls (348.64,50.66) and (377.66,94.69) .. (377.66,149) .. controls (377.66,203.31) and (348.64,247.33) .. (312.83,247.33) .. controls (277.03,247.33) and (248,203.31) .. (248,149) -- cycle ;
\draw    (256,100.33) -- (367.66,100.1) ;
\draw    (248.67,139.95) -- (377.67,139.95) ;
\draw  [fill={rgb, 255:red, 0; green, 75; blue, 245 }  ,fill opacity=1 ] (158,175.04) .. controls (158,172.62) and (159.96,170.66) .. (162.39,170.66) .. controls (164.81,170.66) and (166.78,172.62) .. (166.78,175.04) .. controls (166.78,177.47) and (164.81,179.43) .. (162.39,179.43) .. controls (159.96,179.43) and (158,177.47) .. (158,175.04) -- cycle ;
\draw  [fill={rgb, 255:red, 0; green, 0; blue, 0 }  ,fill opacity=1 ] (321.61,190) .. controls (321.61,187.58) and (323.58,185.61) .. (326,185.61) .. controls (328.42,185.61) and (330.39,187.58) .. (330.39,190) .. controls (330.39,192.42) and (328.42,194.39) .. (326,194.39) .. controls (323.58,194.39) and (321.61,192.42) .. (321.61,190) -- cycle ;





\draw [line width=0.75]    (166.78,175.04) -- (303.61,214.72) ;
\draw  [fill={rgb, 255:red, 0; green, 0; blue, 0 }  ,fill opacity=1 ] (303.61,214.72) .. controls (303.61,212.29) and (305.57,210.33) .. (308,210.33) .. controls (310.42,210.33) and (312.39,212.29) .. (312.39,214.72) .. controls (312.39,217.14) and (310.42,219.11) .. (308,219.11) .. controls (305.57,219.11) and (303.61,217.14) .. (303.61,214.72) -- cycle ;
\draw [line width=0.75]    (312.39,214.72) -- (326,190) ;
\draw [line width=0.75]    (326,185.61) .. controls (357.75,149) and (147.66,59.77) .. (151.44,68.67) ;
\draw  [fill={rgb, 255:red, 0; green, 0; blue, 0 }  ,fill opacity=1 ] (142.67,68.67) .. controls (142.67,66.24) and (144.63,64.28) .. (147.06,64.28) .. controls (149.48,64.28) and (151.44,66.24) .. (151.44,68.67) .. controls (151.44,71.09) and (149.48,73.06) .. (147.06,73.06) .. controls (144.63,73.06) and (142.67,71.09) .. (142.67,68.67) -- cycle ;

\draw (36,70.4) node [anchor=north west][inner sep=0.75pt]    {\large$T_{1}$};
\draw (228,71.4) node [anchor=north west][inner sep=0.75pt]    {\large$T_{2}$};
\draw (31,123.4) node [anchor=north west][inner sep=0.75pt]    {\large$A_{1}$};
\draw (223,108.4) node [anchor=north west][inner sep=0.75pt]    {\large$A_{2}$};
\draw (34,174.4) node [anchor=north west][inner sep=0.75pt]    {\large$Z_{1}$};
\draw (226,168.4) node [anchor=north west][inner sep=0.75pt]    {\large$Z_{2}$};
\draw (146,167.4) node [anchor=north west][inner sep=0.75pt]    {\large$v$};

\draw (328,193.4) node [anchor=north west][inner sep=0.75pt]    {\large$z$};

\draw (312,215.4) node [anchor=north west][inner sep=0.75pt]    {\large$u'$};

\draw (131.06,68.46) node [anchor=north west][inner sep=0.75pt]    {\large$t$};

\end{tikzpicture} & \tikzset{every picture/.style={line width=0.75pt}} 

\begin{tikzpicture}[x=0.75pt,y=0.75pt,yscale=-1,xscale=1]

\draw   (61,152.83) .. controls (61,96.78) and (90.03,51.33) .. (125.83,51.33) .. controls (161.64,51.33) and (190.66,96.78) .. (190.66,152.83) .. controls (190.66,208.89) and (161.64,254.33) .. (125.83,254.33) .. controls (90.03,254.33) and (61,208.89) .. (61,152.83) -- cycle ;
\draw    (63.01,130.53) -- (190,131) ;
\draw    (61,152.83) -- (189.78,152.28) ;
\draw   (248,149) .. controls (248,94.69) and (277.03,50.66) .. (312.83,50.66) .. controls (348.64,50.66) and (377.66,94.69) .. (377.66,149) .. controls (377.66,203.31) and (348.64,247.33) .. (312.83,247.33) .. controls (277.03,247.33) and (248,203.31) .. (248,149) -- cycle ;
\draw    (256,100.33) -- (367.66,100.1) ;
\draw    (248.67,139.95) -- (377.67,139.95) ;
\draw  [fill={rgb, 255:red, 0; green, 75; blue, 245 }  ,fill opacity=1 ] (158,175.04) .. controls (158,172.62) and (159.96,170.66) .. (162.39,170.66) .. controls (164.81,170.66) and (166.78,172.62) .. (166.78,175.04) .. controls (166.78,177.47) and (164.81,179.43) .. (162.39,179.43) .. controls (159.96,179.43) and (158,177.47) .. (158,175.04) -- cycle ;
\draw  [fill={rgb, 255:red, 0; green, 0; blue, 0 }  ,fill opacity=1 ] (321.61,190) .. controls (321.61,187.58) and (323.58,185.61) .. (326,185.61) .. controls (328.42,185.61) and (330.39,187.58) .. (330.39,190) .. controls (330.39,192.42) and (328.42,194.39) .. (326,194.39) .. controls (323.58,194.39) and (321.61,192.42) .. (321.61,190) -- cycle ;
\draw [line width=1]    (166.78,175.04) -- (292.61,123.72) ;
\draw [line width=1]    (326,190) .. controls (317.23,177.45) and (139.66,78.77) .. (143.44,87.67) ;
\draw  [fill={rgb, 255:red, 0; green, 0; blue, 0 }  ,fill opacity=1 ] (134.67,87.67) .. controls (134.67,85.24) and (136.63,83.28) .. (139.06,83.28) .. controls (141.48,83.28) and (143.44,85.24) .. (143.44,87.67) .. controls (143.44,90.09) and (141.48,92.06) .. (139.06,92.06) .. controls (136.63,92.06) and (134.67,90.09) .. (134.67,87.67) -- cycle ;
\draw  [fill={rgb, 255:red, 0; green, 0; blue, 0 }  ,fill opacity=1 ] (292.61,123.72) .. controls (292.61,121.29) and (294.57,119.33) .. (297,119.33) .. controls (299.42,119.33) and (301.39,121.29) .. (301.39,123.72) .. controls (301.39,126.14) and (299.42,128.11) .. (297,128.11) .. controls (294.57,128.11) and (292.61,126.14) .. (292.61,123.72) -- cycle ;
\draw [line width=1]    (297,123.72) -- (326,190) ;

\draw (36,70.4) node [anchor=north west][inner sep=0.75pt]    {\large$T_{1}$};
\draw (228,71.4) node [anchor=north west][inner sep=0.75pt]    {\large$T_{2}$};
\draw (31,123.4) node [anchor=north west][inner sep=0.75pt]    {\large$A_{1}$};
\draw (223,108.4) node [anchor=north west][inner sep=0.75pt]    {\large$A_{2}$};
\draw (34,174.4) node [anchor=north west][inner sep=0.75pt]    {\large$Z_{1}$};
\draw (226,163.4) node [anchor=north west][inner sep=0.75pt]    {\large$Z_{2}$};
\draw (146,167.4) node [anchor=north west][inner sep=0.75pt]    {\large$v$};

\draw (290,105.4) node [anchor=north west][inner sep=0.75pt]    {\large$u$};

\draw (310,195.4) node [anchor=north west][inner sep=0.75pt]    {\large$z_u$};

\draw (141.06,95.46) node [anchor=north west][inner sep=0.75pt]    {\large$t$};

\end{tikzpicture} \\
           \end{tabular}}}
    \caption{Illustration of the paths from Claim \ref{claim:cliques1} 
    }
    \label{fig:claim3cliques}
\end{figure}
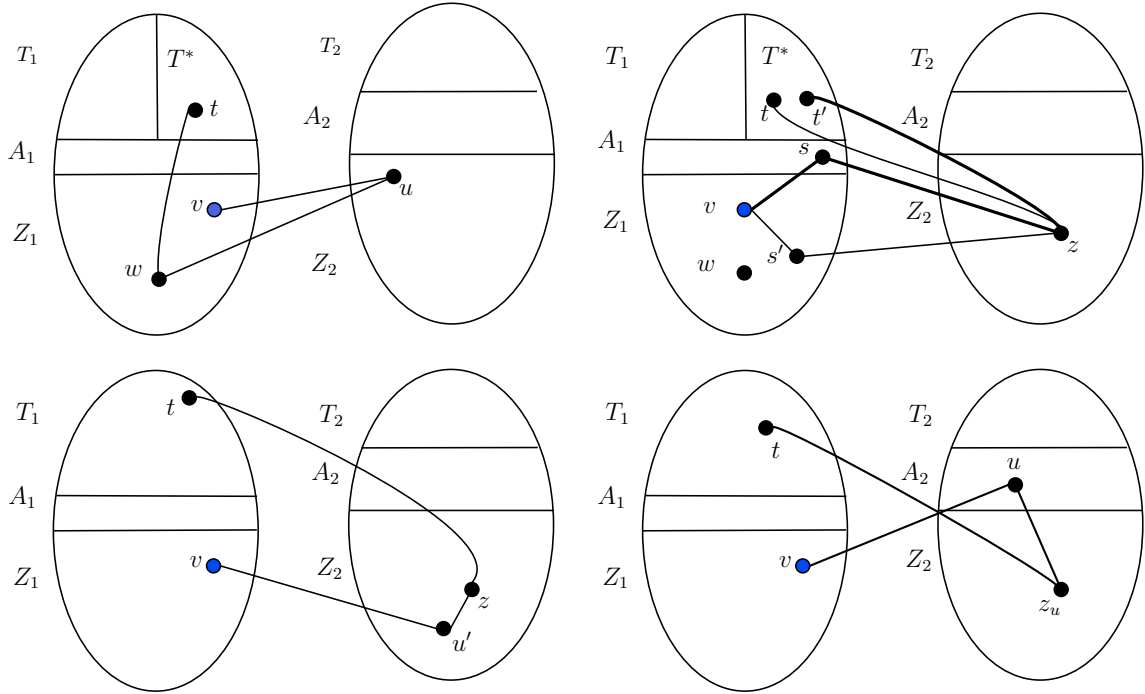

At this point consider the sets $\mathcal{P}_v, T^*$ again, noting that we have so far only added to $\mathcal{P}_v$ those paths which are single edges from $v$ to $T_1\cup T_2$. We consider all these paths to be \emph{type 1} paths. At this point we fix, for each $v\in Z_i$ $i\in\{1, 2\}$, some $z_v\in Z_j$ ($j\neq i$) and some $b_v\in B_j$. We now add to $\mathcal{P}_v$ all of the following acceptable paths, in the following order (see Figure \ref{fig:oddpaths}):
\begin{itemize}
\item {\bf Type 2 path:} $(v, v', b_v, t)$ where $t\in T_i^*$, $v'\in Z_i\cup C_i$
\item {\bf Type 3 path:} $(v, z', z_v, t)$ where $t\in T_i^*$, $z'\in Z_j\setminus \{z_v\}$
\item {\bf Type 4 path:} $(v, x, x', t)$ where $t\in T_i^*$, $x\in C_j$, $x'\in N_{Z_j}(x)$ 
\end{itemize}

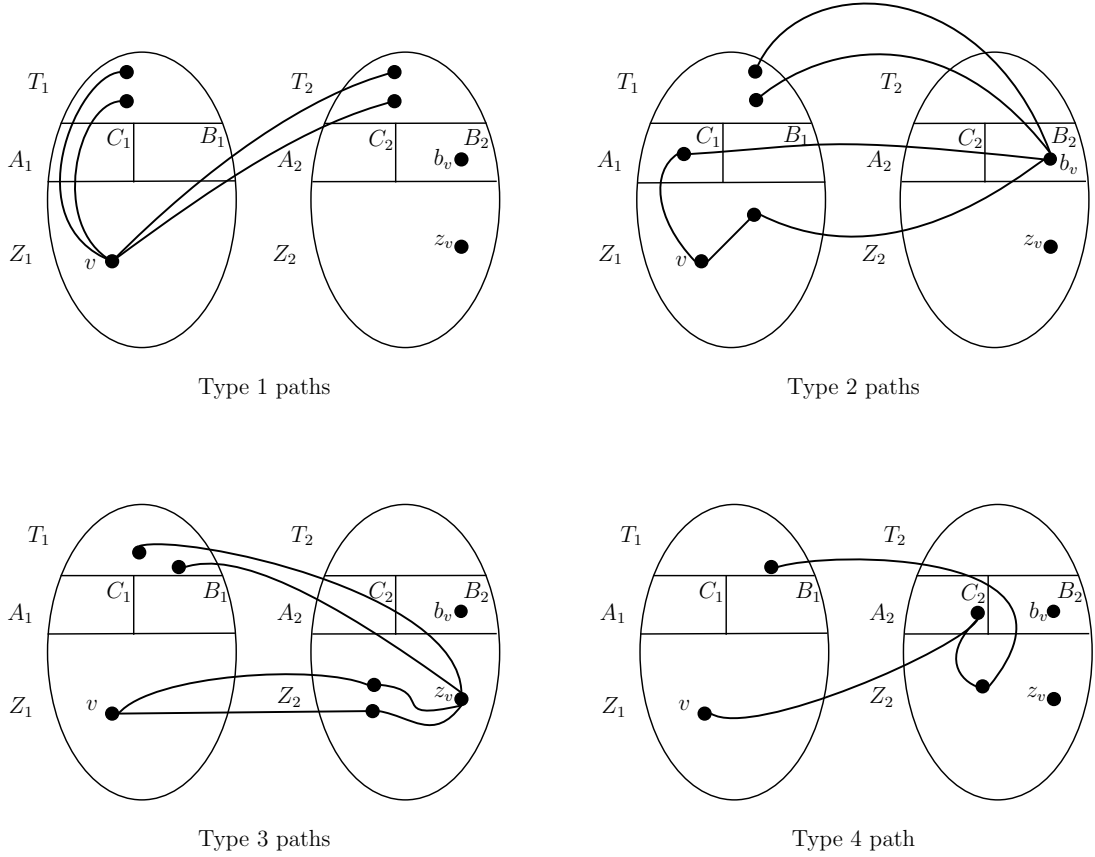
\begin{figure}[h!]
    \centering
    \makebox[\textwidth][c]{%
    \resizebox{1\linewidth}{!}{%
    \begin{tabular}{c c}
        \tikzset{every picture/.style={line width=0.75pt}} 

\begin{tikzpicture}[x=0.75pt,y=0.75pt,yscale=-1,xscale=1,scale=1]
\useasboundingbox (20,-20) rectangle (400,270);

\draw   (61,152.83) .. controls (61,96.78) and (90.03,51.33) .. (125.83,51.33) .. controls (161.64,51.33) and (190.66,96.78) .. (190.66,152.83) .. controls (190.66,208.89) and (161.64,254.33) .. (125.83,254.33) .. controls (90.03,254.33) and (61,208.89) .. (61,152.83) -- cycle ;
\draw    (70,100.22) -- (181,100.22) ;
\draw    (61,140.33) -- (189.78,139.78) ;
\draw   (242,152.83) .. controls (242,96.78) and (271.03,51.33) .. (306.83,51.33) .. controls (342.64,51.33) and (371.66,96.78) .. (371.66,152.83) .. controls (371.66,208.89) and (342.64,254.33) .. (306.83,254.33) .. controls (271.03,254.33) and (242,208.89) .. (242,152.83) -- cycle ;
\draw    (251,100.33) -- (362.66,100.1) ;
\draw    (243,140.33) -- (369.66,140.1) ;
\draw  [fill={rgb, 255:red, 0; green, 0; blue, 0 }  ,fill opacity=1 ] (101,195.04) .. controls (101,192.62) and (102.96,190.66) .. (105.39,190.66) .. controls (107.81,190.66) and (109.78,192.62) .. (109.78,195.04) .. controls (109.78,197.47) and (107.81,199.43) .. (105.39,199.43) .. controls (102.96,199.43) and (101,197.47) .. (101,195.04) -- cycle ;
\draw  [fill={rgb, 255:red, 0; green, 0; blue, 0 }  ,fill opacity=1 ] (341,125.04) .. controls (341,122.62) and (342.96,120.66) .. (345.39,120.66) .. controls (347.81,120.66) and (349.78,122.62) .. (349.78,125.04) .. controls (349.78,127.47) and (347.81,129.43) .. (345.39,129.43) .. controls (342.96,129.43) and (341,127.47) .. (341,125.04) -- cycle ;
\draw  [fill={rgb, 255:red, 0; green, 0; blue, 0 }  ,fill opacity=1 ] (341,185.04) .. controls (341,182.62) and (342.96,180.66) .. (345.39,180.66) .. controls (347.81,180.66) and (349.78,182.62) .. (349.78,185.04) .. controls (349.78,187.47) and (347.81,189.43) .. (345.39,189.43) .. controls (342.96,189.43) and (341,187.47) .. (341,185.04) -- cycle ;
\draw    (300.22,99.89) -- (300.22,140.89) ;
\draw    (120.22,99.89) -- (120.22,140.89) ;
\draw  [fill={rgb, 255:red, 0; green, 0; blue, 0 }  ,fill opacity=1 ] (111,65.04) .. controls (111,62.62) and (112.96,60.66) .. (115.39,60.66) .. controls (117.81,60.66) and (119.78,62.62) .. (119.78,65.04) .. controls (119.78,67.47) and (117.81,69.43) .. (115.39,69.43) .. controls (112.96,69.43) and (111,67.47) .. (111,65.04) -- cycle ;
\draw  [fill={rgb, 255:red, 0; green, 0; blue, 0 }  ,fill opacity=1 ] (111,85.04) .. controls (111,82.62) and (112.96,80.66) .. (115.39,80.66) .. controls (117.81,80.66) and (119.78,82.62) .. (119.78,85.04) .. controls (119.78,87.47) and (117.81,89.43) .. (115.39,89.43) .. controls (112.96,89.43) and (111,87.47) .. (111,85.04) -- cycle ;
\draw [line width=1pt]   (111,85.04) .. controls (81,87.04) and (61.88,170.87) .. (105.39,195.04) ;
\draw [line width=1pt]   (111,65.04) .. controls (81,67.04) and (37.88,168.87) .. (105.39,195.04) ;

\draw  [fill={rgb, 255:red, 0; green, 0; blue, 0 }  ,fill opacity=1 ] (295,65.04) .. controls (295,62.62) and (296.96,60.66) .. (299.39,60.66) .. controls (301.81,60.66) and (303.78,62.62) .. (303.78,65.04) .. controls (303.78,67.47) and (301.81,69.43) .. (299.39,69.43) .. controls (296.96,69.43) and (295,67.47) .. (295,65.04) -- cycle ;

\draw  [fill={rgb, 255:red, 0; green, 0; blue, 0 }  ,fill opacity=1 ] (295,85.04) .. controls (295,82.62) and (296.96,80.66) .. (299.39,80.66) .. controls (301.81,80.66) and (303.78,82.62) .. (303.78,85.04) .. controls (303.78,87.47) and (301.81,89.43) .. (299.39,89.43) .. controls (296.96,89.43) and (295,87.47) .. (295,85.04) -- cycle ;

\draw [color=black,line width=1]   (105.39,195.04) .. controls (180,140) and (240,100) .. (299.39,85.04) ;

\draw [color=black,line width=1]   (105.39,195.04) .. controls (170,130) and (240,80) .. (299.39,65.04) ;

\draw (46,64.4) node [anchor=north west][inner sep=0.75pt]    {\large$T_{1}$};
\draw (227,64.4) node [anchor=north west][inner sep=0.75pt]    {\large$T_{2}$};
\draw (32,118) node [anchor=north west][inner sep=0.75pt]    {\large$A_{1}$};
\draw (217,118) node [anchor=north west][inner sep=0.75pt]    {\large$A_{2}$};
\draw (33,182.4) node [anchor=north west][inner sep=0.75pt]    {\large$Z_{1}$};
\draw (214,182.4) node [anchor=north west][inner sep=0.75pt]    {\large$Z_{2}$};
\draw (325,116.4) node [anchor=north west][inner sep=0.75pt]    {\large$b_v$};
\draw (325,175.4) node [anchor=north west][inner sep=0.75pt]    {\large$z_v$};
\draw (85,192) node [anchor=north west][inner sep=0.75pt]    {\large$v$};
\draw (100,102.4) node [anchor=north west][inner sep=0.75pt]    {\large$C_{1}$};
\draw (164,101.4) node [anchor=north west][inner sep=0.75pt]    {\large$B_{1}$};
\draw (280,103.4) node [anchor=north west][inner sep=0.75pt]    {\large$C_{2}$};
\draw (345,102.4) node [anchor=north west][inner sep=0.75pt]    {\large$B_{2}$};

\end{tikzpicture} & \tikzset{every picture/.style={line width=0.75pt}} 

\begin{tikzpicture}[x=0.75pt,y=0.75pt,yscale=-1,xscale=1,scale=1]
\useasboundingbox (20,-20) rectangle (400,270);

\draw   (61,152.83) .. controls (61,96.78) and (90.03,51.33) .. (125.83,51.33) .. controls (161.64,51.33) and (190.66,96.78) .. (190.66,152.83) .. controls (190.66,208.89) and (161.64,254.33) .. (125.83,254.33) .. controls (90.03,254.33) and (61,208.89) .. (61,152.83) -- cycle ;
\draw    (70,100.22) -- (181,100.22) ;
\draw    (61.83,141.17) -- (190.61,140.61) ;
\draw   (242,152.83) .. controls (242,96.78) and (271.03,51.33) .. (306.83,51.33) .. controls (342.64,51.33) and (371.66,96.78) .. (371.66,152.83) .. controls (371.66,208.89) and (342.64,254.33) .. (306.83,254.33) .. controls (271.03,254.33) and (242,208.89) .. (242,152.83) -- cycle ;
\draw    (251,100.33) -- (362.66,100.1) ;
\draw    (243,140.33) -- (369.66,140.1) ;
\draw  [fill={rgb, 255:red, 0; green, 0; blue, 0 }  ,fill opacity=1 ] (101,195.04) .. controls (101,192.62) and (102.96,190.66) .. (105.39,190.66) .. controls (107.81,190.66) and (109.78,192.62) .. (109.78,195.04) .. controls (109.78,197.47) and (107.81,199.43) .. (105.39,199.43) .. controls (102.96,199.43) and (101,197.47) .. (101,195.04) -- cycle ;
\draw  [fill={rgb, 255:red, 0; green, 0; blue, 0 }  ,fill opacity=1 ] (341,124.77) .. controls (341,122.5) and (342.84,120.66) .. (345.11,120.66) .. controls (347.38,120.66) and (349.22,122.5) .. (349.22,124.77) .. controls (349.22,127.04) and (347.38,128.88) .. (345.11,128.88) .. controls (342.84,128.88) and (341,127.04) .. (341,124.77) -- cycle ;
\draw  [fill={rgb, 255:red, 0; green, 0; blue, 0 }  ,fill opacity=1 ] (341,185.04) .. controls (341,182.62) and (342.96,180.66) .. (345.39,180.66) .. controls (347.81,180.66) and (349.78,182.62) .. (349.78,185.04) .. controls (349.78,187.47) and (347.81,189.43) .. (345.39,189.43) .. controls (342.96,189.43) and (341,187.47) .. (341,185.04) -- cycle ;
\draw    (300.22,99.89) -- (300.22,140.89) ;
\draw    (120.22,99.89) -- (120.22,140.89) ;
\draw [line width=1pt]   (109.78,195.04) -- (141.59,163.08) ;
\draw [line width=1pt]   (89,121.33) .. controls (70.78,133.33) and (72.78,165.33) .. (101,195.04) ;
\draw  [fill={rgb, 255:red, 0; green, 0; blue, 0 }  ,fill opacity=1 ] (89,121.33) .. controls (89,118.91) and (90.96,116.94) .. (93.39,116.94) .. controls (95.81,116.94) and (97.78,118.91) .. (97.78,121.33) .. controls (97.78,123.76) and (95.81,125.72) .. (93.39,125.72) .. controls (90.96,125.72) and (89,123.76) .. (89,121.33) -- cycle ;
\draw  [fill={rgb, 255:red, 0; green, 0; blue, 0 }  ,fill opacity=1 ] (137.2,163.08) .. controls (137.2,160.65) and (139.17,158.69) .. (141.59,158.69) .. controls (144.01,158.69) and (145.98,160.65) .. (145.98,163.08) .. controls (145.98,165.5) and (144.01,167.47) .. (141.59,167.47) .. controls (139.17,167.47) and (137.2,165.5) .. (137.2,163.08) -- cycle ;
\draw [line width=1pt]   (341,125.04) .. controls (322.78,137.04) and (237,208.67) .. (145.98,163.08) ;
\draw [line width=1pt]   (341,125.04) .. controls (184.61,108.06) and (188.07,115.57) .. (97.78,121.33) ;
\draw [line width=1pt]   (345.11,120.66) .. controls (297.72,-5.72) and (169.39,-4.33) .. (142.39,64.67) ;
\draw [line width=1pt]   (345.11,120.66) .. controls (263.72,15.28) and (173,57.28) .. (143,84.28) ;
\draw  [fill={rgb, 255:red, 0; green, 0; blue, 0 }  ,fill opacity=1 ] (138,64.67) .. controls (138,62.24) and (139.96,60.28) .. (142.39,60.28) .. controls (144.81,60.28) and (146.78,62.24) .. (146.78,64.67) .. controls (146.78,67.09) and (144.81,69.06) .. (142.39,69.06) .. controls (139.96,69.06) and (138,67.09) .. (138,64.67) -- cycle ;
\draw  [fill={rgb, 255:red, 0; green, 0; blue, 0 }  ,fill opacity=1 ] (138.61,84.28) .. controls (138.61,81.85) and (140.58,79.89) .. (143,79.89) .. controls (145.42,79.89) and (147.39,81.85) .. (147.39,84.28) .. controls (147.39,86.7) and (145.42,88.67) .. (143,88.67) .. controls (140.58,88.67) and (138.61,86.7) .. (138.61,84.28) -- cycle ;

\draw (46,64.4) node [anchor=north west][inner sep=0.75pt]    {\large$T_{1}$};
\draw (227,64.4) node [anchor=north west][inner sep=0.75pt]    {\large$T_{2}$};
\draw (32,118) node [anchor=north west][inner sep=0.75pt]    {\large$A_{1}$};
\draw (217,118) node [anchor=north west][inner sep=0.75pt]    {\large$A_{2}$};
\draw (33,182.4) node [anchor=north west][inner sep=0.75pt]    {\large$Z_{1}$};
\draw (214,182.4) node [anchor=north west][inner sep=0.75pt]    {\large$Z_{2}$};
\draw (350,120) node [anchor=north west][inner sep=0.75pt]    {\large$b_v$};
\draw (325,175.4) node [anchor=north west][inner sep=0.75pt]    {\large$z_v$};
\draw (86,190) node [anchor=north west][inner sep=0.75pt]    {\large$v$};
\draw (100,101.4) node [anchor=north west][inner sep=0.75pt]    {\large$C_{1}$};
\draw (160,100.4) node [anchor=north west][inner sep=0.75pt]    {\large$B_{1}$};
\draw (280,102.4) node [anchor=north west][inner sep=0.75pt]    {\large$C_{2}$};
\draw (344,101.4) node [anchor=north west][inner sep=0.75pt]    {\large$B_{2}$};

\end{tikzpicture} \\
        {\large Type 1 paths} & {\large Type 2 paths}\\
        \tikzset{every picture/.style={line width=0.75pt}} 

\begin{tikzpicture}[x=0.75pt,y=0.75pt,yscale=-1,xscale=1,scale=1]
\useasboundingbox (20,-20) rectangle (400,270);

\draw   (61,152.83) .. controls (61,96.78) and (90.03,51.33) .. (125.83,51.33) .. controls (161.64,51.33) and (190.66,96.78) .. (190.66,152.83) .. controls (190.66,208.89) and (161.64,254.33) .. (125.83,254.33) .. controls (90.03,254.33) and (61,208.89) .. (61,152.83) -- cycle ;
\draw    (70,100.22) -- (181,100.22) ;
\draw    (61,140.33) -- (189.78,139.78) ;
\draw   (242,152.83) .. controls (242,96.78) and (271.03,51.33) .. (306.83,51.33) .. controls (342.64,51.33) and (371.66,96.78) .. (371.66,152.83) .. controls (371.66,208.89) and (342.64,254.33) .. (306.83,254.33) .. controls (271.03,254.33) and (242,208.89) .. (242,152.83) -- cycle ;
\draw    (251,100.33) -- (362.66,100.1) ;
\draw    (243,140.33) -- (369.66,140.1) ;
\draw  [fill={rgb, 255:red, 0; green, 0; blue, 0 }  ,fill opacity=1 ] (101,195.04) .. controls (101,192.62) and (102.96,190.66) .. (105.39,190.66) .. controls (107.81,190.66) and (109.78,192.62) .. (109.78,195.04) .. controls (109.78,197.47) and (107.81,199.43) .. (105.39,199.43) .. controls (102.96,199.43) and (101,197.47) .. (101,195.04) -- cycle ;
\draw  [fill={rgb, 255:red, 0; green, 0; blue, 0 }  ,fill opacity=1 ] (341,124.77) .. controls (341,122.5) and (342.84,120.66) .. (345.11,120.66) .. controls (347.38,120.66) and (349.22,122.5) .. (349.22,124.77) .. controls (349.22,127.04) and (347.38,128.88) .. (345.11,128.88) .. controls (342.84,128.88) and (341,127.04) .. (341,124.77) -- cycle ;
\draw  [fill={rgb, 255:red, 0; green, 0; blue, 0 }  ,fill opacity=1 ] (341,185.04) .. controls (341,182.62) and (342.96,180.66) .. (345.39,180.66) .. controls (347.81,180.66) and (349.78,182.62) .. (349.78,185.04) .. controls (349.78,187.47) and (347.81,189.43) .. (345.39,189.43) .. controls (342.96,189.43) and (341,187.47) .. (341,185.04) -- cycle ;
\draw    (300.22,99.89) -- (300.22,140.89) ;
\draw    (120.22,99.89) -- (120.22,140.89) ;
\draw [line width=1pt]   (281,175.33) .. controls (247.44,160.83) and (132.58,167.96) .. (109.78,195.04) ;
\draw  [fill={rgb, 255:red, 0; green, 0; blue, 0 }  ,fill opacity=1 ] (281,175.33) .. controls (281,172.91) and (282.96,170.94) .. (285.39,170.94) .. controls (287.81,170.94) and (289.78,172.91) .. (289.78,175.33) .. controls (289.78,177.76) and (287.81,179.72) .. (285.39,179.72) .. controls (282.96,179.72) and (281,177.76) .. (281,175.33) -- cycle ;


\draw [line width=1pt]   (345.39,189.43) .. controls (292.22,201.44) and (329.22,179.44) .. (289.78,175.33) ;

\draw [line width=1pt]   (345.39,180.66) .. controls (344.12,111.34) and (155.21,70.94) .. (124,79.89) ;

\draw  [fill={rgb, 255:red, 0; green, 0; blue, 0 }  ,fill opacity=1 ] (119.61,84.28) .. controls (119.61,81.85) and (121.58,79.89) .. (124,79.89) .. controls (126.42,79.89) and (128.39,81.85) .. (128.39,84.28) .. controls (128.39,86.7) and (126.42,88.67) .. (124,88.67) .. controls (121.58,88.67) and (119.61,86.7) .. (119.61,84.28) -- cycle ;
\draw  [fill={rgb, 255:red, 0; green, 0; blue, 0 }  ,fill opacity=1 ] (280,193.33) .. controls (280,190.91) and (281.96,188.94) .. (284.39,188.94) .. controls (286.81,188.94) and (288.78,190.91) .. (288.78,193.33) .. controls (288.78,195.76) and (286.81,197.72) .. (284.39,197.72) .. controls (281.96,197.72) and (280,195.76) .. (280,193.33) -- cycle ;
\draw [line width=1pt]   (345.39,189.43) .. controls (326.37,215) and (311.37,197.23) .. (288.78,193.33) ;
\draw [line width=1pt]   (109.78,195.04) -- (280,193.33) ;
\draw [line width=1pt]   (345.39,180.66) .. controls (250.99,111.84) and (190.37,83.23) .. (155.6,94.33) ;
\draw  [fill={rgb, 255:red, 0; green, 0; blue, 0 }  ,fill opacity=1 ] (146.82,94.33) .. controls (146.82,91.91) and (148.79,89.94) .. (151.21,89.94) .. controls (153.63,89.94) and (155.6,91.91) .. (155.6,94.33) .. controls (155.6,96.76) and (153.63,98.72) .. (151.21,98.72) .. controls (148.79,98.72) and (146.82,96.76) .. (146.82,94.33) -- cycle ;

\draw (46,64.4) node [anchor=north west][inner sep=0.75pt]    {\large $T_{1}$};
\draw (227,64.4) node [anchor=north west][inner sep=0.75pt]    {\large$T_{2}$};
\draw (32,118) node [anchor=north west][inner sep=0.75pt]    {\large$A_{1}$};
\draw (217,118) node [anchor=north west][inner sep=0.75pt]    {\large$A_{2}$};
\draw (33,182.4) node [anchor=north west][inner sep=0.75pt]    {\large$Z_{1}$};
\draw (217,175.4) node [anchor=north west][inner sep=0.75pt]    {\large$Z_{2}$};
\draw (325,116.4) node [anchor=north west][inner sep=0.75pt]    {\large$b_v$};
\draw (325,175.4) node [anchor=north west][inner sep=0.75pt]    {\large$z_v$};
\draw (86,184.4) node [anchor=north west][inner sep=0.75pt]    {\large$v$};
\draw (100,103.4) node [anchor=north west][inner sep=0.75pt]    {\large$C_{1}$};
\draw (166,103.4) node [anchor=north west][inner sep=0.75pt]    {\large$B_{1}$};
\draw (280,103.4) node [anchor=north west][inner sep=0.75pt]    {\large$C_{2}$};
\draw (345,103.4) node [anchor=north west][inner sep=0.75pt]    {\large$B_{2}$};

\end{tikzpicture} & \tikzset{every picture/.style={line width=0.75pt}} 

\begin{tikzpicture}[x=0.75pt,y=0.75pt,yscale=-1,xscale=1,scale=1]
\useasboundingbox (20,-20) rectangle (400,270);

\draw   (61,152.83) .. controls (61,96.78) and (90.03,51.33) .. (125.83,51.33) .. controls (161.64,51.33) and (190.66,96.78) .. (190.66,152.83) .. controls (190.66,208.89) and (161.64,254.33) .. (125.83,254.33) .. controls (90.03,254.33) and (61,208.89) .. (61,152.83) -- cycle ;
\draw    (70,100.22) -- (181,100.22) ;
\draw    (61,140.33) -- (189.78,139.78) ;
\draw   (242,152.83) .. controls (242,96.78) and (271.03,51.33) .. (306.83,51.33) .. controls (342.64,51.33) and (371.66,96.78) .. (371.66,152.83) .. controls (371.66,208.89) and (342.64,254.33) .. (306.83,254.33) .. controls (271.03,254.33) and (242,208.89) .. (242,152.83) -- cycle ;
\draw    (251,100.33) -- (362.66,100.1) ;
\draw    (243,140.33) -- (369.66,140.1) ;
\draw  [fill={rgb, 255:red, 0; green, 0; blue, 0 }  ,fill opacity=1 ] (101,195.04) .. controls (101,192.62) and (102.96,190.66) .. (105.39,190.66) .. controls (107.81,190.66) and (109.78,192.62) .. (109.78,195.04) .. controls (109.78,197.47) and (107.81,199.43) .. (105.39,199.43) .. controls (102.96,199.43) and (101,197.47) .. (101,195.04) -- cycle ;
\draw  [fill={rgb, 255:red, 0; green, 0; blue, 0 }  ,fill opacity=1 ] (341,124.77) .. controls (341,122.5) and (342.84,120.66) .. (345.11,120.66) .. controls (347.38,120.66) and (349.22,122.5) .. (349.22,124.77) .. controls (349.22,127.04) and (347.38,128.88) .. (345.11,128.88) .. controls (342.84,128.88) and (341,127.04) .. (341,124.77) -- cycle ;
\draw  [fill={rgb, 255:red, 0; green, 0; blue, 0 }  ,fill opacity=1 ] (341,185.04) .. controls (341,182.62) and (342.96,180.66) .. (345.39,180.66) .. controls (347.81,180.66) and (349.78,182.62) .. (349.78,185.04) .. controls (349.78,187.47) and (347.81,189.43) .. (345.39,189.43) .. controls (342.96,189.43) and (341,187.47) .. (341,185.04) -- cycle ;
\draw    (300.22,99.89) -- (300.22,140.89) ;
\draw    (120.22,99.89) -- (120.22,140.89) ;

\draw  [fill={rgb, 255:red, 0; green, 0; blue, 0 }  ,fill opacity=1 ] (288.67,125.83) .. controls (288.67,123.4) and (290.63,121.44) .. (293.05,121.44) .. controls (295.48,121.44) and (297.44,123.4) .. (297.44,125.83) .. controls (297.44,128.25) and (295.48,130.22) .. (293.05,130.22) .. controls (290.63,130.22) and (288.67,128.25) .. (288.67,125.83) -- cycle ;



\draw [line width=1pt] (300.78,176.33) .. controls (379,81.33) and (190.37,83.23) .. (155.6,94.33) ;
\draw  [fill={rgb, 255:red, 0; green, 0; blue, 0 }  ,fill opacity=1 ] (146.82,94.33) .. controls (146.82,91.91) and (148.79,89.94) .. (151.21,89.94) .. controls (153.63,89.94) and (155.6,91.91) .. (155.6,94.33) .. controls (155.6,96.76) and (153.63,98.72) .. (151.21,98.72) .. controls (148.79,98.72) and (146.82,96.76) .. (146.82,94.33) -- cycle ;
\draw [line width=1pt]   (293.05,130.22) .. controls (282.05,150.22) and (138.33,212.89) .. (109.78,195.04) ;

\draw  [fill={rgb, 255:red, 0; green, 0; blue, 0 }  ,fill opacity=1 ] (292,176.33) .. controls (292,173.91) and (293.96,171.94) .. (296.39,171.94) .. controls (298.81,171.94) and (300.78,173.91) .. (300.78,176.33) .. controls (300.78,178.76) and (298.81,180.72) .. (296.39,180.72) .. controls (293.96,180.72) and (292,178.76) .. (292,176.33) -- cycle ;

\draw [line width=1pt]   (293.05,130.22) .. controls (270.06,151.93) and (277.06,169.93) .. (292,176.33) ;

\draw (46,64.4) node [anchor=north west][inner sep=0.75pt]    {\large$T_{1}$};
\draw (227,64.4) node [anchor=north west][inner sep=0.75pt]    {\large$T_{2}$};
\draw (32,118) node [anchor=north west][inner sep=0.75pt]    {\large$A_{1}$};
\draw (217,118) node [anchor=north west][inner sep=0.75pt]    {\large$A_{2}$};
\draw (33,182.4) node [anchor=north west][inner sep=0.75pt]    {\large$Z_{1}$};
\draw (217,175.4) node [anchor=north west][inner sep=0.75pt]    {\large$Z_{2}$};
\draw (327,116.4) node [anchor=north west][inner sep=0.75pt]    {\large$b_v$};
\draw (325,175.4) node [anchor=north west][inner sep=0.75pt]    {\large$z_v$};
\draw (86,184.4) node [anchor=north west][inner sep=0.75pt]    {\large$v$};
\draw (100,103.4) node [anchor=north west][inner sep=0.75pt]    {\large$C_{1}$};
\draw (166,103.4) node [anchor=north west][inner sep=0.75pt]    {\large$B_{1}$};
\draw (280,104.4) node [anchor=north west][inner sep=0.75pt]    {\large$C_{2}$};
\draw (346,103.4) node [anchor=north west][inner sep=0.75pt]    {\large$B_{2}$};

\end{tikzpicture}\\
        {\large Type 3 paths} & {\large Type 4 path}
    \end{tabular}}}
    \caption{Four types of odd paths in $\mathcal{P}_v$ that begin at $v\in Z_1$ and end in $T_1$.}
    \label{fig:oddpaths}
\end{figure}

After adding paths of types 1 through 4 to each $\mathcal{P}_v$, we get the following. We will use the notation that for some vertex $v$ and set $X\subseteq V(G)$, $d_X(v)$ is the number of neighbors of $v$ in $X$.

\begin{claim} \label{claim:Pvsize}Let $\{i, j\}=\{1, 2\}$. For any $v\in Z_i$, 
$$ t-1>|\mathcal{P}_v| \geq d(v)-|B_j|-1.$$ 
\end{claim}
\begin{proofc} Suppose without loss that that $i=1$ and let $v\in Z_1$. Note that within $\mathcal{P}_v$ we have $d_{T_1}(v)+|T_2|$ paths of type 1. We also have $d_{A_1}(v)+|Z_1|-1$ paths of type 2, since $G[Z_1]$ is a clique (by Claim \ref{claim:cliques1}) and hence $v$ has $|Z_1|-1$ neighbors in $Z_1$. In $\mathcal{P}_v$, we have $|Z_2|-1$ paths of type 3, since $G[Z_2]$ is a clique (by Claim \ref{claim:cliques1}) . And we have $|C_2|$ paths of type 4 in $\mathcal{P}_v$. So 
$$|\mathcal{P}_v|\geq d_{t_1}(v)+|T_2|+d_{A_1}(v)+|Z_1|-1+|Z_2|-1+|C_2|=d(v)-|B_2|-1.$$
The upper bound on $|\mathcal{P}_v|$ is immediate since $G$ has no strong odd $K_t$-immersion.\end{proofc}

We now have the following.

\begin{claim} \label{claim:NoSolved} For every $v\in Z_1\cup Z_2$, the set $\mathcal{P}_{v}$ is currently maximal (i.e. there are no acceptable paths that can be added to it).
\end{claim}

\begin{proofc} Suppose not, and suppose without loss that $\mathcal{P}_{v_1}$ is not currently maximal for some $v_1\in Z_1$. Then we can add an additional acceptable path to $\mathcal{P}_{v_1}$, after which we have the following,
$$t-2\geq |\mathcal{P}_{v_1}| \geq d(v_1)-|B_j|$$ 
according to Claim \ref{claim:Pvsize} and the fact that $G$ has no strong odd $K_t$-immersion (recall $|T_1\cup T_2|=t-1$). Choosing any $v_2\in Z_2$, Claim \ref{claim:Pvsize} tells us that
$$t-2 \geq|\mathcal{P}_{v_2}| \geq d(v_2)-|B_j|-1.$$ 
Adding the two inequalities together and using the fact that $\delta(G)\geq \tfrac{3}{2}(t-1)$ we get:
$$2t-4 \geq 3(t-1)-|B_1|-|B_2|-1
\hspace*{.2in}\Rightarrow\hspace*{.2in} |B_1|+|B_2|\geq t.$$
However this is a contradiction, since $G[B_1\cup B_2]$ is a clique by Claim \ref{claim:cliques1}.
\end{proofc}

Up to this point we have needed to consider $\mathcal{P}_v$ for vertices $v$ in both $G_1$ and $G_2$; from here onwards we will be able to lose this symmetry and just focus on $G_1$.

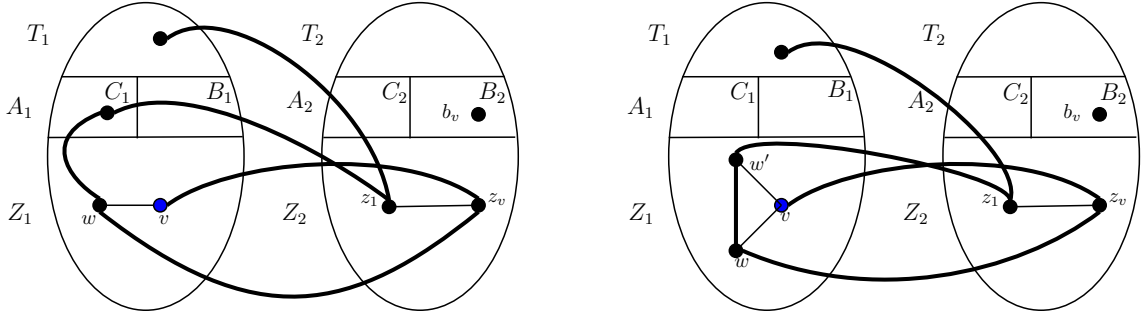
\begin{figure}[h]
\centering
\begin{tabular}{cc}
  \resizebox{0.5\linewidth}{!}{\tikzset{every picture/.style={line width=0.75pt}}

\begin{tikzpicture}[x=0.75pt,y=0.75pt,yscale=-1,xscale=1]
\useasboundingbox (20,-20) rectangle (400,270);

\draw   (61,152.83) .. controls (61,96.78) and (90.03,51.33) .. (125.83,51.33) .. controls (161.64,51.33) and (190.66,96.78) .. (190.66,152.83) .. controls (190.66,208.89) and (161.64,254.33) .. (125.83,254.33) .. controls (90.03,254.33) and (61,208.89) .. (61,152.83) -- cycle ;
\draw    (70,100.22) -- (181,100.22) ;
\draw    (61,140.33) -- (189.78,139.78) ;
\draw   (242,152.83) .. controls (242,96.78) and (271.03,51.33) .. (306.83,51.33) .. controls (342.64,51.33) and (371.66,96.78) .. (371.66,152.83) .. controls (371.66,208.89) and (342.64,254.33) .. (306.83,254.33) .. controls (271.03,254.33) and (242,208.89) .. (242,152.83) -- cycle ;
\draw    (251,100.33) -- (362.66,100.1) ;
\draw    (243,140.33) -- (369.66,140.1) ;
\draw  [fill={rgb, 255:red, 0; green, 0; blue, 250 }  ,fill opacity=1 ] (131,185.04) .. controls (131,182.62) and (132.96,180.66) .. (135.39,180.66) .. controls (137.81,180.66) and (139.78,182.62) .. (139.78,185.04) .. controls (139.78,187.47) and (137.81,189.43) .. (135.39,189.43) .. controls (132.96,189.43) and (131,187.47) .. (131,185.04) -- cycle ;
\draw  [fill={rgb, 255:red, 0; green, 0; blue, 0 }  ,fill opacity=1 ] (341,125.04) .. controls (341,122.62) and (342.96,120.66) .. (345.39,120.66) .. controls (347.81,120.66) and (349.78,122.62) .. (349.78,125.04) .. controls (349.78,127.47) and (347.81,129.43) .. (345.39,129.43) .. controls (342.96,129.43) and (341,127.47) .. (341,125.04) -- cycle ;
\draw  [fill={rgb, 255:red, 0; green, 0; blue, 0 }  ,fill opacity=1 ] (341,185.04) .. controls (341,182.62) and (342.96,180.66) .. (345.39,180.66) .. controls (347.81,180.66) and (349.78,182.62) .. (349.78,185.04) .. controls (349.78,187.47) and (347.81,189.43) .. (345.39,189.43) .. controls (342.96,189.43) and (341,187.47) .. (341,185.04) -- cycle ;
\draw    (300.22,99.89) -- (300.22,140.89) ;
\draw    (120.22,99.89) -- (120.22,140.89) ;
\draw  [fill={rgb, 255:red, 0; green, 0; blue, 0 }  ,fill opacity=1 ] (96,124.04) .. controls (96,121.62) and (97.96,119.66) .. (100.39,119.66) .. controls (102.81,119.66) and (104.78,121.62) .. (104.78,124.04) .. controls (104.78,126.47) and (102.81,128.43) .. (100.39,128.43) .. controls (97.96,128.43) and (96,126.47) .. (96,124.04) -- cycle ;
\draw  [fill={rgb, 255:red, 0; green, 0; blue, 0 }  ,fill opacity=1 ] (91,185.04) .. controls (91,182.62) and (92.96,180.66) .. (95.39,180.66) .. controls (97.81,180.66) and (99.78,182.62) .. (99.78,185.04) .. controls (99.78,187.47) and (97.81,189.43) .. (95.39,189.43) .. controls (92.96,189.43) and (91,187.47) .. (91,185.04) -- cycle ;
\draw  [fill={rgb, 255:red, 0; green, 0; blue, 0 }  ,fill opacity=1 ] (282,186.04) .. controls (282,183.62) and (283.96,181.66) .. (286.39,181.66) .. controls (288.81,181.66) and (290.78,183.62) .. (290.78,186.04) .. controls (290.78,188.47) and (288.81,190.43) .. (286.39,190.43) .. controls (283.96,190.43) and (282,188.47) .. (282,186.04) -- cycle ;
\draw    (286.39,186.04) -- (345.39,185.04) ;
\draw    (131,185.04) -- (99.78,185.04) ;
\draw [line width=2pt]   (139.78,185.04) .. controls (179.78,155.04) and (295,144.61) .. (345.39,180.66) ;
\draw [line width=2pt]   (96,124.04) .. controls (56.76,139.69) and (72.76,166.69) .. (95.39,180.66) ;
\draw [line width=2pt]   (95.39,189.43) .. controls (198.67,270.33) and (261.67,257.33) .. (345.39,189.43) ;
\draw [line width=2pt]   (104.78,124.04) .. controls (153.61,98.72) and (245.61,150.72) .. (286.39,181.66) ;
\draw  [fill={rgb, 255:red, 0; green, 0; blue, 0 }  ,fill opacity=1 ] (131,75.04) .. controls (131,72.62) and (132.96,70.66) .. (135.39,70.66) .. controls (137.81,70.66) and (139.78,72.62) .. (139.78,75.04) .. controls (139.78,77.47) and (137.81,79.43) .. (135.39,79.43) .. controls (132.96,79.43) and (131,77.47) .. (131,75.04) -- cycle ;
\draw [line width=2pt]   (139.78,75.04) .. controls (179.78,45.04) and (274.61,113.72) .. (286.39,181.66) ;

\draw (46,64.4) node [anchor=north west][inner sep=0.75pt]    {\large $T_{1}$};
\draw (227,64.4) node [anchor=north west][inner sep=0.75pt]    {\large$T_{2}$};
\draw (32,112.4) node [anchor=north west][inner sep=0.75pt]    {\large$A_{1}$};
\draw (217,108.4) node [anchor=north west][inner sep=0.75pt]    {\large$A_{2}$};
\draw (33,182.4) node [anchor=north west][inner sep=0.75pt]    {\large$Z_{1}$};
\draw (214,182.4) node [anchor=north west][inner sep=0.75pt]    {\large$Z_{2}$};
\draw (320,116.4) node [anchor=north west][inner sep=0.75pt]    {$b_{v}$};
\draw (350,175.4) node [anchor=north west][inner sep=0.75pt]    {$z_{v}$};
\draw (133,190) node [anchor=north west][inner sep=0.75pt]    {$v$};
\draw (97,102.4) node [anchor=north west][inner sep=0.75pt]    {\large$C_{1}$};
\draw (164,102.4) node [anchor=north west][inner sep=0.75pt]    {\large$B_{1}$};
\draw (280,102.4) node [anchor=north west][inner sep=0.75pt]    {\large$C_{2}$};
\draw (344,102.4) node [anchor=north west][inner sep=0.75pt]    {\large$B_{2}$};
\draw (82,190) node [anchor=north west][inner sep=0.75pt]    {$w$};
\draw (266,176.4) node [anchor=north west][inner sep=0.75pt]    {$z_1$};

\end{tikzpicture}} &
  \resizebox{0.5\linewidth}{!}{\tikzset{every picture/.style={line width=0.75pt}}

\begin{tikzpicture}[x=0.75pt,y=0.75pt,yscale=-1,xscale=1,scale=1]
\useasboundingbox (20,-20) rectangle (400,270);

\draw   (61,152.83) .. controls (61,96.78) and (90.03,51.33) .. (125.83,51.33) .. controls (161.64,51.33) and (190.66,96.78) .. (190.66,152.83) .. controls (190.66,208.89) and (161.64,254.33) .. (125.83,254.33) .. controls (90.03,254.33) and (61,208.89) .. (61,152.83) -- cycle ;
\draw    (70,100.22) -- (181,100.22) ;
\draw    (61,140.33) -- (189.78,139.78) ;
\draw   (242,152.83) .. controls (242,96.78) and (271.03,51.33) .. (306.83,51.33) .. controls (342.64,51.33) and (371.66,96.78) .. (371.66,152.83) .. controls (371.66,208.89) and (342.64,254.33) .. (306.83,254.33) .. controls (271.03,254.33) and (242,208.89) .. (242,152.83) -- cycle ;
\draw    (251,100.33) -- (362.66,100.1) ;
\draw    (243,140.33) -- (369.66,140.1) ;
\draw  [fill={rgb, 255:red, 0; green, 0; blue, 250 }  ,fill opacity=1 ] (131,185.04) .. controls (131,182.62) and (132.96,180.66) .. (135.39,180.66) .. controls (137.81,180.66) and (139.78,182.62) .. (139.78,185.04) .. controls (139.78,187.47) and (137.81,189.43) .. (135.39,189.43) .. controls (132.96,189.43) and (131,187.47) .. (131,185.04) -- cycle ;
\draw  [fill={rgb, 255:red, 0; green, 0; blue, 0 }  ,fill opacity=1 ] (341,125.04) .. controls (341,122.62) and (342.96,120.66) .. (345.39,120.66) .. controls (347.81,120.66) and (349.78,122.62) .. (349.78,125.04) .. controls (349.78,127.47) and (347.81,129.43) .. (345.39,129.43) .. controls (342.96,129.43) and (341,127.47) .. (341,125.04) -- cycle ;
\draw  [fill={rgb, 255:red, 0; green, 0; blue, 0 }  ,fill opacity=1 ] (341,185.04) .. controls (341,182.62) and (342.96,180.66) .. (345.39,180.66) .. controls (347.81,180.66) and (349.78,182.62) .. (349.78,185.04) .. controls (349.78,187.47) and (347.81,189.43) .. (345.39,189.43) .. controls (342.96,189.43) and (341,187.47) .. (341,185.04) -- cycle ;
\draw    (300.22,99.89) -- (300.22,140.89) ;
\draw    (120.22,99.89) -- (120.22,140.89) ;
\draw  [fill={rgb, 255:red, 0; green, 0; blue, 0 }  ,fill opacity=1 ] (131,84.04) .. controls (131,81.62) and (132.96,79.66) .. (135.39,79.66) .. controls (137.81,79.66) and (139.78,81.62) .. (139.78,84.04) .. controls (139.78,86.47) and (137.81,88.43) .. (135.39,88.43) .. controls (132.96,88.43) and (131,86.47) .. (131,84.04) -- cycle ;
\draw  [fill={rgb, 255:red, 0; green, 0; blue, 0 }  ,fill opacity=1 ] (101,155.04) .. controls (101,152.62) and (102.96,150.66) .. (105.39,150.66) .. controls (107.81,150.66) and (109.78,152.62) .. (109.78,155.04) .. controls (109.78,157.47) and (107.81,159.43) .. (105.39,159.43) .. controls (102.96,159.43) and (101,157.47) .. (101,155.04) -- cycle ;
\draw  [fill={rgb, 255:red, 0; green, 0; blue, 0 }  ,fill opacity=1 ] (101,215.04) .. controls (101,212.62) and (102.96,210.66) .. (105.39,210.66) .. controls (107.81,210.66) and (109.78,212.62) .. (109.78,215.04) .. controls (109.78,217.47) and (107.81,219.43) .. (105.39,219.43) .. controls (102.96,219.43) and (101,217.47) .. (101,215.04) -- cycle ;
\draw  [fill={rgb, 255:red, 0; green, 0; blue, 0 }  ,fill opacity=1 ] (282,186.04) .. controls (282,183.62) and (283.96,181.66) .. (286.39,181.66) .. controls (288.81,181.66) and (290.78,183.62) .. (290.78,186.04) .. controls (290.78,188.47) and (288.81,190.43) .. (286.39,190.43) .. controls (283.96,190.43) and (282,188.47) .. (282,186.04) -- cycle ;
\draw    (286.39,186.04) -- (345.39,185.04) ;
\draw    (105.39,215.04) -- (135.39,185.04) ;
\draw    (135.39,185.04) -- (105.39,155.04) ;
\draw [line width=2pt]   (105.39,210.66) -- (105.39,159.43) ;
\draw [line width=2pt]   (139.78,185.04) .. controls (179.78,155.04) and (295,144.61) .. (345.39,180.66) ;
\draw [line width=2pt]   (109.78,215.04) .. controls (195.39,252) and (291,230.78) .. (345.39,189.43) ;
\draw [line width=2pt]   (105.39,150.66) .. controls (117.88,130.71) and (284.44,162.89) .. (286.39,181.66) ;
\draw [line width=2pt]   (139.78,84.04) .. controls (179.78,54.04) and (297.56,148.56) .. (286.39,181.66) ;

\draw (46,64.4) node [anchor=north west][inner sep=0.75pt]    {\large$T_{1}$};
\draw (227,64.4) node [anchor=north west][inner sep=0.75pt]    {\large$T_{2}$};
\draw (32,112.4) node [anchor=north west][inner sep=0.75pt]    {\large$A_{1}$};
\draw (217,108.4) node [anchor=north west][inner sep=0.75pt]    {\large$A_{2}$};
\draw (33,182.4) node [anchor=north west][inner sep=0.75pt]    {\large$Z_{1}$};
\draw (214,182.4) node [anchor=north west][inner sep=0.75pt]    {\large$Z_{2}$};
\draw (320,116.4) node [anchor=north west][inner sep=0.75pt]    {$b_{v}$};
\draw (350,175.4) node [anchor=north west][inner sep=0.75pt]    {$z_{v}$};
\draw (133,188.44) node [anchor=north west][inner sep=0.75pt]    {$v$};
\draw (100,102.4) node [anchor=north west][inner sep=0.75pt]    {\large$C_{1}$};
\draw (164,101.4) node [anchor=north west][inner sep=0.75pt]    {\large$B_{1}$};
\draw (280,103.4) node [anchor=north west][inner sep=0.75pt]    {\large$C_{2}$};
\draw (344,103.4) node [anchor=north west][inner sep=0.75pt]    {\large$B_{2}$};
\draw (112.78,150) node [anchor=north west][inner sep=0.75pt]    {$w'$};
\draw (103,218.44) node [anchor=north west][inner sep=0.75pt]    {$w$};
\draw (266,176.4) node [anchor=north west][inner sep=0.75pt]    {$z_1$};

\end{tikzpicture}} 
\end{tabular}
\caption{Paths from the proofs of Claims \ref{claim:NoNe1} and \ref{claim:size2}.}
\label{fig:claims789}
\end{figure}

\begin{claim}\label{claim:NoNe1} $A_1=B_1$. 
\end{claim}
\begin{proofc} Suppose there exists an edge $aw\in E(G)$ with $a\in A_1, w\in Z_1$. Choose any vertices $v\in Z_1\setminus \{w\}$ and $z_1\in Z_2\setminus z_v$; we know such vertices exist by Claim \ref{claim:Ztwo}. See the top-left picture in Figure \ref{fig:claims789}. Note that $z_1\sim v, w$ and $z\sim a, t$ for any $t\in T^*$. So $P= (v , z_v , w , a , z_1 ,t)$ is an odd path for any $t\in T^*$. By comparing $P$ to paths of types 1-4 in Figure \ref{fig:oddpaths} we can see that $P$ is acceptable for $\mathcal{P}_{v}$; this is a contradiction to Claim \ref{claim:NoSolved}.
\end{proofc}

\begin{claim} \label{claim:size2} $|Z_1|=2$. 
\end{claim}
\begin{proofc}
    Suppose not. By Claim \ref{claim:Ztwo} this means that $|Z_1|\geq 3$, and we can choose three vertices $v, w, w'\in Z_1$. By Claim \ref{claim:Ztwo} we can also choose $z\in Z_2 \setminus z_v$.  See the top-right picture in Figure \ref{fig:claims789}. Then $z_v\sim w, z$, and $z\sim w', t$ for any $t\in T^*$. Moreover, $w\sim w'$ because $G[Z_1]$ is a clique (by Claim \ref{claim:cliques1}). Hence $P=(v , z_v , w, w',  z ,t)$ is an odd path  from $v$ to some $t \in T^*$, and looking at the paths of type 1-4 in Figure \ref{fig:oddpaths} compared with the path exhibited in the top-right picture in Figure \ref{fig:claims789}, we see that $P$ is acceptable for $\mathcal{P}_v$, contradicting Claim \ref{claim:NoSolved}.
 \end{proofc}

Let us now label the two vertices in $Z_1$, say $Z_1=\{v, w\}$. Since $G$ has no strong odd $K_t$-immersion, we know we must have some vertex $t\in T^*$ for $\mathcal{P}_v$. We know that $v\not\sim t$ as otherwise we would have added the edge $vt$ to $\mathcal{P}_v$ as a type 1 path.

Consider now the graph $H=G_1\setminus\{v, w, t\}$. Since $Z_1=\{v, w\}$ we know that $H=G_1[T_1\cup A_1]=G_1[M_1]\setminus\{t\}$, where $M_1 \subseteq X_1$ was chosen as an inclusion-wise minimal subset of $X_1$ such that $G[M_1]$ still contains $K_{t_1}$ as a strong odd immersion. So in particular, $H$ contains no strong odd $K_{t_1}$-immersion. By minimality of our $G$, and since $\alpha(H)\leq \alpha(G)\leq 2$, this means that $\chi(H)\leq \lceil\tfrac{3}{2}(t_1-1)\rceil=\tfrac{3}{2}(t_1-1)$, with the last equality due to the fact that $t_1$ is odd. Recall that $\chi(G_1)=\tfrac{3}{2}t_1+\tfrac{1}{2}$, so $\chi(G_1)-\chi(H)\geq 2$. In fact, we must have equality here, since any coloring of $H$ could be extended to $G_1$ using only two new colors: simply use one new color on $t, v$, and one other new color on $w$.

Let $\varphi$ be a $\chi(H)$-coloring of $H$, and consider the colors represented on the vertices of $N_{T_1}(w)$. If there is some color of $\varphi$ missing from this set, say color 1, then color one can be used on $w$, and some new color (say color $0$) can be used on both $v$ and $t$ to get a $(\chi(H)+1)$-coloring of $G_1$, contradicting $\chi(G_1)-\chi(H)=2$. So in fact $|N_{T_1}(w)|\geq \chi(H)$. Observe that the vertices of $N_{T_1}(w)$ are the terminals of some strong odd $K_{\chi(H)}$-immersion in $H$. Hence $G_1$ contains strong odd $K_{\chi(H)+1}$-immersion with terminals $N_{T_1}(w)\cup\{w\}$. However then $\chi(H)+1\leq t_1$ by definition of $t_1$, and since $\chi(G_1)-\chi(H)= 2$, this tells us that $\chi(G_1)\leq t_1+1$. Using $\chi(G_1)=\tfrac{3}{2}t_1+\tfrac{1}{2}$, this implies that $t_1\leq 3$. Then $G_1$ has no strong odd $K_{t_1+1}$-immersion, and it is known that $\chi(G_1)<t_1+1$. However now, since
$\chi(G_1)=\tfrac{3}{2}t_1+\tfrac{1}{2}$, this is a contradiction. \end{proof}

As we mentioned in the introduction, we don't believe the $\tfrac{3}{2}$ in the above theorem is optimal (in contrast to the situation for odd-minors), and indeed we conjecture that a better fraction can be found. Our proof method seems very much limited to $\tfrac{3}{2}$ however -- in particular, this is essential for Claim \ref{claim:NoSolved} and consequently for Claim \ref{claim:size2}, without which we are not able to get a better coloring for $G_1$.

We observe that $\alpha=2$ is used sparingly in our proof of Theorem \ref{thm:main}, but it is needed in order to apply Theorem \ref{thm:VergStrong}, and also crucially within the proof of Claim \ref{claim:cliques1}. For fixed but larger values of $\alpha$, we would need a different approach. When $\alpha=3$, we can use the statement of Theorem \ref{thm:main} to get the following modest improvement on the afore-mentioned result of Bustamante et al.

\begin{theorem} Let $t\in\mathbb{Z}^+$ and let $G$ be a graph. If $\alpha(G)=3$ and $G$ has no strong odd $K_t$-immersion, then $\chi(G)\leq 4(t-1)$.
\end{theorem}

\begin{proof} Let $G$ be a vertex-minimal counterexample, so in particular
$\alpha=\alpha(G)= 3$, $G$ has no strong odd $K_t$-immersion, and $k:= \chi(G)> 4(t-1)$. We may assume that $t\geq 5$ as otherwise the result holds. By the theorem of Bustamante et al., we get $\lfloor \frac{n}{4.5}\rfloor \leq t$ with $n:=|V(G)|$. 

We first suppose that $\lfloor \frac{n}{4.5}\rfloor \leq 3$. Then $n\leq 17$. If $t\geq 6$, then $17 \geq n\geq \chi(G)> 4(t-1)\geq 20$, contradiction. So we may suppose that $t=5$ and $\chi(G)> 4(t-1)= 16$ with $n=17$. But since a non-complete graph on 17 vertices is 16-colorable, we get a contradiction.  

We may now assume that $\lfloor \frac{n}{4.5}\rfloor \geq 4$. Setting $\ell=\tfrac{n}{4.5}$ we note that $\tfrac{9}{16}\ell\leq \lfloor\ell\rfloor-1$ (since $\lfloor\ell\rfloor\geq 4$), so we get that
 $\frac{n}{8}\leq\lfloor \frac{n}{4.5}\rfloor-1 \leq t-1$. This tells us that $\frac{n}{2} \leq 4(t-1)$, from which we get $\frac{n}{2}\leq k -1$ and  $n\leq 2k-2$. Note that $G$ is $k$-critical by minimality. Since $n\leq 2k-2$ we apply Theorem \ref{Gallai} and get a partition of $V(G)$ into non-empty sets $X_1, X_2$ such that $x_1 x_2 \in E(G)$ for all $x_1 \in X_1$ and $x_2 \in X_2$. Let $G_1 := G[X_1]$ and $G_2 := G[X_2]$, and note that we have $\chi(G) = \chi(G_1) + \chi(G_2)$. For $i \in \{1, 2\}$, let $t_i$ denote the largest positive integer such that $G_i$ has $K_{t_i}$ as a strong odd immersion; note that $G$ has a strong odd immersion of $K_{t_1+t_2}$, so $t_1+t_2\leq t-1$. Let $\alpha_i := \alpha(G_i) \leq \alpha \leq 3$. If $\alpha_i = 3$ we get $\chi(G_i)\leq 4t_i$ (since $G_i$ does not have $K_{t_i+1}$ as strong odd immersion) by the minimality of $G$, and otherwise we get this (and in fact an even better result) by Theorem \ref{thm:main}. Hence we get the following:
$$\chi(G) = \chi(G_1) + \chi(G_2) \leq 4t_1+4t_2 = 4 (t_1+t_2)\leq 4(t-1)<\chi(G),$$
which is a contradiction.
\end{proof}

\section*{Appendix}

\begin{apptheorem}\label{thm:VergStrong}
    Every $n$-vertex graph $G$ with $\alpha(G)\leq 2$ has a strong odd $K_{\lceil n/3\rceil}$-immersion.
\end{apptheorem}

\begin{proof}
    Suppose not, and choose a counterexample $G$ with $n$ minimum; $G$ is an $n$-vertex graph with $\alpha(G)\leq 2$ that has no strong odd $K_{\lceil n/3\rceil}$-immersion. Denote the non-neighborhood of any $x\in V(G)$ by $\overline{N}_G(x)=\{y| y\not\sim x, y\neq x, y\in V(G)\}$; note that every $G[\overline{N}_G(x)]$ is a clique since $\alpha(G)\leq  2$.

    \begin{claim*}
        $\delta(G) \ge \lfloor \frac{2n}{3}\rfloor$
    \end{claim*}
    \begin{proofc}
        Suppose, on the contrary, there is a vertex $x\in V(G)$ such that $d(x) \leq  \lfloor \frac{2n}{3}\rfloor -1$. Then 
       $$
        |\overline{N}_G(x)| \ge  n  - 1- (\left \lfloor \tfrac{2n}{3} \right \rfloor -1) 
        \geq n - \lfloor\tfrac{2n}{3}\rfloor
        = \lceil\tfrac{n}{3}\rceil,      $$
    which is a contradiction since $G[\overline{N}_G(x)]$ is a clique. 
    \end{proofc}

 Since $G$ is not a clique we can choose $\{u,v\}$ an independent set of size two in $G$. Consider $G' = G -\{u,v\}$. Note that $G'$ has a strong odd $K_{\lceil (n-2)/3\rceil}$-immersion by minimality; let $T$ be the set of terminals for some such immersion. We will aim to increase the size of our immersion by one by adding $v$ to this set of terminals; note that this would give our desired contradiction since $\left \lceil \tfrac{n-2}{3} \right \rceil +1 \ge \left \lceil \tfrac{n}{3} \right \rceil$.

Let $G'' = G' - T$ and consider the set $N_{G''}(u) \cap N_{G''}(v)$. We must find an odd path from $v$ to every vertex $t\in T\setminus N_T(v)$. Since $\alpha(G)\leq 2$ and $u\not\sim v$, we know that $u\sim t$ for all such $t$. So each vertex $w \in N_{G''}(u) \cap N_{G''}(v)$ provides the odd path $(v, w, u, t)$ from $v$ to $t$; these paths are clearly edge-disjoint when $w, t$ are both different. So it must be the case that $|N_{G''}(u) \cap N_{G''}(v)| < |T\setminus N_T(v)|.$ 

Note that  $|N_{G''}(u) \cap N_{G''}(v)| =d(v) - d_T(v) - \overline{N}_{G''}(u) $, since every non-neighbor of $u$ in $G''$ must be a neighbor of $v$ (as  $\alpha(G)\leq 2$ and $u\not\sim v$). Since $|T|=\lceil \tfrac{n-2}{3} \rceil$ and $d(v)\geq \lfloor \frac{2n}{3}\rfloor$, the conclusion of the last paragraph gives us
$$  \lfloor\tfrac{2n}{3}\rfloor  - d_T(v) - \overline{N}_{G''}(u) <\left \lceil \tfrac{n-2}{3} \right \rceil - d_T(v),$$
and hence
    \begin{equation}\label{eq:clique} \left \lfloor \tfrac{2n}{3} \right \rfloor - \left \lceil \tfrac{n-2}{3} \right \rceil +1 \leq  \overline{N}_{G''}(u).\end{equation}
The vertices of $\overline{N}_{G''}(u) \cup \{v\}$ induce a clique, since $u\not\sim v$ and $\alpha(G)\leq 2$. The quantity on the left-hand side of (\ref{eq:clique}) is: $\tfrac{2n}{3}-\tfrac{n}{3}+1>\lceil\tfrac{n}{3}\rceil$ when $n\equiv 0$ (mod 3); $\tfrac{2n-2}{3}-\tfrac{n-1}{3}+1=\lceil\tfrac{n}{3}\rceil$ when $n\equiv 1$ (mod 3), and; $\tfrac{2n-1}{3}-\tfrac{n-2}{3}+1>\lceil\tfrac{n}{3}\rceil$ when $n\equiv 2$ (mod 3). In all these cases we find a clique of size at least $\lceil\tfrac{n}{3}\rceil$ in $G$, which is a contradiction.
\end{proof}

\end{document}